\documentclass[amscd,amssymb,verbatim,11pt]{amsart}
\usepackage[fontsize = 10bp]{fontsize}
\usepackage{scalerel}
\usepackage{stackengine,wasysym}
\usepackage{bbm}

\usepackage{amssymb,latexsym, amsmath, amscd, array, graphicx, pb-diagram}
\usepackage{hyperref}
\usepackage{color}

\usepackage{tikz-cd}

\usepackage{color}
\hypersetup{
    colorlinks=true,
    linkcolor=red,
    urlcolor=blue,
    citecolor=blue
           }

\usepackage{amsthm}
\usepackage{hyperref}
\usepackage[all]{xy}
\usepackage[scr]{rsfso}

\swapnumbers

\theoremstyle{plain}

\newtheorem{thm}{Theorem}[section]
\newtheorem{theorem}[thm]{Theorem}

\newtheorem{lemma}[thm]{Lemma}

\newtheorem{proposition}[thm]{Proposition}
\newtheorem{corollary}[thm]{Corollary}

\theoremstyle{definition}

\newtheorem{definition}[thm]{Definition}

\newtheorem{remark}[thm]{Remark}

\newtheorem{ex}[thm]{Example}

\newtheorem{question}[thm]{Question}

 \newcommand{\Wi}{\widetilde}

\DeclareMathOperator{\cat}{{\mathsf{cat}}}
\DeclareMathOperator{\TC}{{\mathsf{TC}}}
\DeclareMathOperator{\secat}{{\mathsf{secat}}}
\DeclareMathOperator{\dsecat}{{\mathsf{dsecat}}}

\DeclareMathOperator{\dcat}{{\mathsf{dcat}}}
\DeclareMathOperator{\dTC}{{\mathsf{dTC}}}

\DeclareMathOperator{\cd}{{\mathsf{cd}}}

\DeclareMathOperator{\Ker}{{\rm Ker}}

\newcommand{\D}{\mathcal{D}}
\DeclareMathOperator{\hdim}{{\mathsf{h-dim}}}
  \newcommand{\F}{\mathcal{F}}

\usepackage{enumitem}

\def\Im{\protect\operatorname{Im}}

\makeatletter
\def\@tocline#1#2#3#4#5#6#7{\relax
  \ifnum #1>\c@tocdepth % then omit
  \else
    \par \addpenalty\@secpenalty\addvspace{#2}%
    \begingroup \hyphenpenalty\@M
    \@ifempty{#4}{%
      \@tempdima\csname r@tocindent\number#1\endcsname\relax
    }{%
      \@tempdima#4\relax
    }%
    \parindent\z@ \leftskip#3\relax \advance\leftskip\@tempdima\relax
    \rightskip\@pnumwidth plus4em \parfillskip-\@pnumwidth
    #5\leavevmode\hskip-\@tempdima
      \ifcase #1
       \or\or \hskip 1em \or \hskip 2em \else \hskip 3em \fi%
      #6\nobreak\relax
    \hfill\hbox to\@pnumwidth{\@tocpagenum{#7}}\par% <---- \dotfill -> \hfill
    \nobreak
    \endgroup
  \fi}
\makeatother

\def\scr{\mathcal}
\def\A{{\scr A}}
\def\B{{\scr B}}
\def\O{{\scr O}}
\def\C{{\mathbb C}}
\def\Z{{\mathbb Z}}
\def\Q{{\mathbb Q}}
\def\R{{\mathbb R}}

\def\rpn{{\R P^n}}

\def\1{\hbox{\rm\rlap {1}\hskip.03in{\rom I}}}
\def\Bbbone{{\rm1\mathchoice{\kern-0.25em}{\kern-0.25em}
{\kern-0.2em}{\kern-0.2em}I}}

\def\wt{\widetilde}

\usepackage{comment}

\long\def\forget#1\forgotten{} %

\newcommand\ver[1]{\marginpar{\tiny Changed in Ver \VER}}

\date{\today}

\begin{document}

\begin{abstract} 
We develop the theory of probabilistic variants of the one-category and diagonal topological complexity, which bound the classical LS-category and topological complexity from below. Unlike any other classical or probabilistic invariants, these invariants are rigid on spaces with finite fundamental group. On Eilenberg--Mac~Lane spaces, we identify these new invariants with distributional category and complexity, respectively, and use them to illuminate aspects of the behavior of the latter invariants on aspherical spaces and products of spaces. We also study their properties on covering maps, $\pi_1$-isomorphisms, $H$-spaces, and closed essential manifolds, and consequently, obtain the first examples of closed manifolds beyond the real projective spaces on which the distributional theory disagrees with the classical one. 
\end{abstract}

% We study their homotopic and algebraic properties 

% We show how these invariants are rigid 

%We study these new invariants on covering spaces, discrete groups, and essential manifolds, and identify them on discrete groups with distributional category and complexity, respectively. 

%We use these new invariants to demystify the behavior of distributional category and complexity on torsion-free groups, finite groups, and products of spaces, and answer nuanced questions in the

\title[Distributional 1-Category and diagonal distributional complexity]{On distributional one-category, 
\\
diagonal distributional complexity, 
\\
and related invariants}

\author[E.~Jauhari]{Ekansh~Jauhari}

\author[J.~Oprea]{John~Oprea}

\address{Ekansh Jauhari, Department of Mathematics, University of Florida, 358 Little Hall, Gainesville, FL 32611, USA.}

\email{ekanshjauhari@ufl.edu}

\address{John Oprea, Department of Mathematics, Cleveland State University, Cleveland, OH 44115, USA.}

\email{jfoprea@gmail.com}



\maketitle
\tableofcontents

\section{Introduction}\label{sec:intro}
The subject of Lusternik--Schnirelmann (LS) category goes back almost 100 years.
The original idea was to understand the complexity of a space (such as a manifold)
by determining the least number of simple pieces that were required to 
assemble it. This number (or nowadays one less than this number) is then
an invariant of (the homotopy type of) the space and serves the usual algebraic
topological function of distinguishing homotopically different spaces. Of course,
Lusternik and Schnirelmann defined category in order to bound the number
of critical points of any smooth function on a manifold from below, and this was essential for their proof of the fact that the topological sphere with any metric
possesses at least 3 closed geodesics. Fox is often credited with bringing
LS-category into the realm of algebraic topology and, in fact, here we will
use a variant of his so-called \emph{one-category}. From the description above,
we see that LS-category is a ``molecular theory'' in the best reductionist 
tradition of classical $19$-th century science --- see~\cite{CLOT} for a modern
approach to LS-category. In the last 20 years or so,
a distinct, but intertwined, theory of \emph{topological complexity} has been
developed by Michael Farber~\cite{Fa} and (now) a host of others. This theory has its
roots in understanding how difficult the motion planning problem in robotics
is for various configuration spaces. This, too, is a molecular theory. Even from 
the vague descriptions we have just given, it should be apparent that
calculating LS-category and topological complexity is a very hard thing 
to do for most topological spaces, so any approximating invariants we can find for them prove to be useful (and sometimes essential).

Recently, some new homotopy invariants of Lusternik--Schnirelmann type were introduced in~\cite{DJ,KW}. 
They are called \emph{distributional category} in~\cite{DJ} and \emph{analog category} in~\cite{KW},
and they provide new lower bounds for LS-category. Both approaches define invariants in terms 
of probability measures on spaces, but different topologies are used in the definitions. 
In many cases, it has now been shown that the two notions agree. Here, we shall always refer to 
distributional category and use the notation $\dcat(X)$ for the distributional category of a space $X$, but we will also refer to~\cite{KW} for various results when they apply. 
Distributional category considers probability measures on the based path
space $P_0(X)$ which lie in a single fiber of the evaluation map to $X$. Of course, 
the fiber of this map is the loop space $\Omega X$, and a probability measure here is a relatively complicated thing. Especially, in the original formulation of 
distributional category (and topological complexity)~\cite{DJ,Ja1}, we see a collection of
dynamical states (i.e., paths) with associated probabilities (adding to one, of course) 
which all collapse to a designated final state, and without carrying the analogy
too far (we hope), this is reminiscent of the Feynman path-integral approach to
Quantum Mechanics. Certainly, we have gone too far here, but the idea of a kind
of superposition of categorical states is not too outlandish. Furthermore, 
at present, there does not seem to be a way of molecularizing the theory by
cutting up a given space into ``atoms''. In this sense, the theories of distributional category
and distributional topological complexity are holistic theories rather than reductionist ones.

In this paper, we go a step further by ``discretizing" in a well-defined sense. 
For a space $X$, we define a new invariant called the \emph{distributional one-category}, 
denoted $\dcat_1(X)$, that instead considers
probability measures on the universal cover $q\colon \widetilde X \to X$ that 
lie in a single fiber of $q$. Of course, the advantage here is that the fiber of a covering is discrete, and this 
entails all sorts of simplifying factors. We prove
that $\dcat_1$ is a lower bound for $\dcat$ and, in various cases, is equal 
to $\dcat$. Because we work with 
discrete fibers, we can be rather explicit in our proofs, and we hope this 
makes the subject more transparent.

As we show in this paper, $\dcat_1$ turns out to be a very useful tool in understanding other distributional invariants. For instance, $\dcat_1$ helps us obtain new examples of manifolds (namely, some Lens spaces and projective product spaces) on which the distributional invariants differ from their classical counterparts. Also, it helps disprove a product inequality for the distributional invariants; the question of the existence of such an inequality (whose analogue in the classical theory exists and is very useful for computations) has been open since the inception of these new invariants. Furthermore, owing to its discrete nature, $\dcat_1$ offers relatively simpler proofs to some key results in the theory of distributional category of aspherical spaces. These features of $\dcat_1$ provide motivation to develop further the theory of distributional one-category and related complexity-type invariants.

\subsection*{Structure of the paper}This paper is organized as follows. After some background material in Section~\ref{sec:prelim}, we formally define distributional one-category ($\dcat_1$) in Section~\ref{sec:dcat1} and elucidate its essential
properties. In particular, we show that for a covering map $f\colon Y \to X$, we have
$\dcat_1(Y) \leq \dcat_1(X)$. In addition, we show that for a discrete group 
$\pi$, we have $\dcat_1(\pi) = \dcat(\pi)$, where $\dcat$ is the previously defined \emph{distributional category} of~\cite{DJ}. This then allows us to analyze $\dcat(\pi)$
from a simpler viewpoint. 

In Section~\ref{sec: relationship with pi1}, we describe cases where $\dcat_1(X)
= \dcat(K(\pi_1(X),1))$, and look at possibilities for $\dcat_1(X)$ and $\dcat_1(\pi_1(X))$ to be arbitrarily far apart. Moreover, we extend a result of~\cite{KW}
by showing that if $\pi_1(X)$ is finite for a space $X$, then $\dcat_1(X) \leq
|\pi_1(X)| - 1$. The powerful influence of the fundamental group on $\dcat_1$ is also
displayed by the fact that $\dcat_1(X) = 2$ for any space $X$ whose fundamental 
group has a single relator (that is not a proper power, of course). 

In Section~\ref{dcat1 of groups}, we use $\dcat_1$ to revisit results about $\dcat(\pi)$ for $\pi$ either a finite group or a torsion-free discrete group. In the latter case, we show an Eilenberg--Ganea-type equality $\dcat_1(\pi)=\cd(\pi)$ in the spirit of~\cite{KW}. While these results are
known, we believe that the $\dcat_1$-approach offers an enlightening perspective. 

In Section~\ref{sec:essential}, we study $\dcat_1$ of essential manifolds~\cite{Gr} and explore cases where $\dcat_1(M)=\dim(M)$ for closed manifolds $M$ (of course, it is not always true). This allows us to extend the above equality beyond the realm of aspherical manifolds. As an application, we compute the distributional one-category of closed $3$-manifolds having torsion-free fundamental group. This section recovers similar results on $\dcat$ of some essential manifolds from~\cite{Ja2}. 

Section~\ref{sec: diag distrib TC} introduces the \emph{sequential distributional diagonal complexity} $\dTC^\D_m(X)$ of a space $X$, an invariant which is intimately related to $\dcat_1$ and serves as a lower bound to the \emph{sequential distributional complexity} $\dTC_m(X)$ of $X$ of~\cite{Ja1}. We establish various properties of $\dTC^\D_m(X)$ and analyze the relationship between the diagonal distributional complexity of a space and its discrete fundamental group. As a bonus, we also show that a product inequality \emph{does not hold} for distributional invariants. This was not 
%able to be 
seen directly using ordinary distributional 
invariants, but is amenable to our techniques. 

We end this paper in Section~\ref{sec:an inequality} by giving an upper bound on $\dcat(X)$ in terms of $\dcat_1(X)$ and $\dcat(\widetilde X)$. This allows us to give the \emph{first examples} of finite-dimensional CW complexes beyond $\rpn$ (where $\dcat(\rpn)
= 1 < n = \cat(\rpn)$ for all $n\ge 2$) with $\dcat \not = \cat$ and $\dTC_m\ne \TC_m$, where $\TC_m$ is the classical $m$-th sequential topological complexity,~\cite{Fa,Ru2}. In some sense, then, this solution of a formerly intractable problem justifies our approach. 

\subsection*{Notations and conventions} All topological spaces considered in this paper are connected CW complexes. We use the term \emph{maps} for continuous functions. The symbol ``$\cong$" is used to denote homeomorphisms, isomorphisms, and bijections, and the symbol ``$\simeq$" is used to denote homotopy equivalences of maps and spaces. 

\section{Preliminaries}\label{sec:prelim}
This section contains some facts about (distributional) sectional category.
\subsection{Sectional category}\label{subsec:secat section}
Suppose $f:E\to B$ is a (Hurewicz) fibration. For each $n \ge 1$, the iterated join of $n$ copies of $E$ along $f$ is
\[
\ast^n E = \left\{\sum_{i=1}^{n} \lambda_{i} e_{i} \hspace{1mm} \middle| \hspace{1mm} e_{i} \in E, \hspace{0.5mm} \sum_{i=1}^{n}\lambda_{i} = 1, \hspace{0.5mm} \lambda_{i} \geq 0, \hspace{0.5mm} f(e_{i})= f(e_{j})\right\},
\]
and the iterated fibration $\ast^nf : \ast^nE \to B$ is defined as
\[
\ast^n f \left(\sum_{i=1}^{n} \lambda_{i} e_{i}\right) = f(e_i)
\]
for any $i$ with $\lambda_{i} > 0$. The smallest integer $n$ such that $\ast^{n+1}f$ has a section is called the \emph{sectional category} of $f$, denoted $\secat(f)$, see~\cite{Sch}. It is well-known that $\secat$ is a homotopy invariant of fibrations.

\begin{remark}\label{rmk:homotopy section}
    If $f:E\to B$ is a fibration that admits a homotopy section, then $f$ admits a true section. This follows from the homotopy lifting property of $f$, see, for example,~\cite[Exercise 1.22]{CLOT}. So, in this paper, we will use the term ``section" for both a homotopy section and a true section of a given fibration since one can be obtained from the other. Similarly, since a lift of a map $g:A\to B$ along a fibration $f$ is essentially a section of the pullback fibration $g^*f$, we will use the term ``lift" for both a homotopy lift as well as a true lift of a given fibration along a given map.
\end{remark}

If $E=P_0(X)$ is the based path space (containing paths \emph{terminating} at a fixed basepoint $x_0 \in X$) and $f$ is the evaluation-at-zero map, defined $e^X(\phi)=\phi(0)$, then the least $n$ such that $\ast^{n+1}e^X$ admits a section is called the \emph{Lusternik--Schnirelmann category} of $X$, denoted $\cat(X)$. More generally, for a map $f:Y\to X$, the least $n$ such that $f$ has a lift along $\ast^{n+1}e^X$ is called the \emph{Lusternik--Schnirelmann category} of $f$, denoted $\cat(f)$. Clearly, $\cat(X)=\cat(1_X)$. Since $\cat$ is a homotopy invariant and the
homotopy type of an Eilenberg--Mac~Lane space $K(\pi,1)$ is determined by $\pi$, we 
may use the notation $\cat(K(\pi,1)) = \cat(\pi)$.

If $E=P(X)$ is the free path space and $f$ is the fibration $\text{ev}_m^X\colon P(X)\to X^m$ defined as 
\[
\text{ev}_m^X(\phi)=\left(\phi(0),\phi\left(\frac{1}{m-1}\right),\ldots,\phi\left(\frac{m-2}{m-1}\right),\phi(1)\right),
\]
then the least $n$ such that $\ast^{n+1} \text{ev}^X_m
$ has a section is called the $m$\emph{-th sequential topological complexity} of $X$, denoted $\TC_m(X)$, see~\cite{Fa,Ru2} for details and the interpretation of this invariant in terms of robot motion planning. 

For a discrete group $\Gamma$, its \textit{cohomological dimension}, denoted $\cd(\Gamma)$, is defined to be the largest integer $k$ such that $H^k(\Gamma,A)\ne 0$ for some $\Z\Gamma$-module $A$. The Eilenberg--Ganea theorem (see~\cite[Theorem VIII.7.1]{Br}) implies $\cat(\pi)=\cd(\pi)$, so that the LS-category of discrete groups has an algebraic description. Moreover, for $\cd(\pi) > 2$, the smallest dimension of a CW complex that is a $K(\pi,1)$ (i.e., the \emph{homotopy dimension}) is $\hdim(K(\pi,1)) = \cd(\pi)$.

\subsection{Distributional sectional category}\label{sec:basic2}

Let $\B(Z)$ denote the set of probability measures on a topological space $Z$ and 
$$\B_n(Z) = \{\mu \in B(Z) \mid\, |\text{supp}(\mu)| \leq n\}$$ 
be the subspace of measures supported by at most $n$ points. So, each such measure $\mu\in\B_n(Z)$ 
can be written as a formal sum 
$$\mu = \sum_{i=1}^n a_i \delta_{z_i} = \sum_{i=1}^{n} a_i z_i,$$
where $z_i \in Z$, $\delta_{z_i}$ is the point (Dirac) measure at $z_i$, and 
$a_i \geq 0$ with $\sum_{i=1}^n a_i = 1$. We identify the space $Z$ with a 
subspace of $\B_n(Z)$ by means of Dirac measures $\delta_z$ for every $z \in Z$.
The support of $\mu$ is the subset $\text{supp}(\mu) = \{z_i\in Z\mid\, a_i > 0\}$. In the case when $Z$ is a metric space, we consider 
the L\'evy-Prokhorov metric on $\B_n(Z)$, see~\cite{Pr,DJ}. A different topology is defined 
in~\cite{KK,KW}, but for the spaces we consider, the results of~\cite{KW} also hold, since~\cite{KW2} shows that the distributional invariants agree with their analog counterparts on compact metric spaces and finite groups, among other spaces.

If $f\colon E \to B$ is a map, then for each integer $n\ge 1$, there is a space
\[
E_n = \bigcup_{b\in B}\B_n\left(f^{-1}(b)\right)=\left\{\mu \in \B_n(E)\mid\, \text{supp}(\mu) \subset f^{-1}(b)\ \text{for\ some\ } b \in B\right\},
\]
which is topologized as a subspace of $\B_n(E)$, and a map $f_n\colon E_n \to B$ defined as
\[
f_n(\mu) = b \ \text{ if } \mu\in\B_n\left(f^{-1}(b)\right).
\]

\begin{remark}[\protect{Naturality}]\label{rem:naturality}
    Suppose we have the following commutative diagram:
\[
\begin{tikzcd}
    X \arrow{r}{f}  \arrow{d}{p}
    &
    Y \arrow{d}{q}
    \\
    B \arrow{r}{g}
    &
    C.
\end{tikzcd}
\]
Of course, for any integer $n\ge 1$, we have maps $p_n\colon X_n\to B$ and $q_n\colon Y_n\to C$ as defined above. Applying the functor $\B_n$ on $f$ yields a map $\B_n(f)\colon \B_n(X)\to \B_n(Y)$. Let $f_n$ be the restriction of $\B_n(f)$ to the space $X_n\subset \B_n(X)$. It is easy to check that the image of $f_n$ lies in $Y_n\subset \B_n(Y)$, and $q_n\circ f_n=g\circ p_n$ --- see, for example, the proof of~\cite[Proposition~5.3]{Ja1}. Hence, given a map $f\colon X\to Y$ over $g\colon B\to C$, we get a natural map $f_n\colon X_n\to Y_n$ also over $g\colon B\to C$. This fact will be used frequently in our proofs.
\end{remark}

If $f\colon E\to B$ is a fibration with fiber $F$, then $f_n\colon E_n \to B$ is a fibration with fiber $\B_n(F)$, see~\cite{DJ,KW}. If $F$ is $r$-connected, then $\B_n(F)$ is $(2n+r-2)$-connected for $r\ge 1$ (see~\cite{KK}), $(n-1)$-connected for $r=0$, and $(n-2)$-connected otherwise, see~\cite[Section 3]{Dr}. The smallest integer $n$ such that $f_{n+1}\colon E_{n+1}\to B$ has a section is called the \emph{distributional sectional category} of $f$, denoted $\dsecat(f)$. If two fibrations are fiberwise homotopic, then their $\dsecat$ values coincide,~\cite{Ja1}. We refer to~\cite[Section 2.1]{Ja2} for a motivation behind this invariant. The similar invariant defined in~\cite{KW} is called 
``analog sectional category" and denoted by $\sf{{asecat}}$$(f)$. 

If $E = P_0(X)$ and $f=e^X$, then we write
$e_n^X \colon P_0(X)_n \to X$, and the least $n$ such that $e_{n+1}^X$ has a section is called the \emph{distributional category} of $X$, denoted $\dcat(X)$, see~\cite{DJ} (a similar invariant was denoted $\sf{{acat}}$$(X)$ in~\cite{KW}). If $f\colon Y \to X$
is a map, then the least $n$ such that $f$ has a lift $\widetilde f$ in
\[
\begin{tikzcd}[contains/.style = {draw=none,"\in" description,sloped}]
&
P_0(X)_{n+1} \arrow{d}{e_{n+1}^X}
\\
Y  \arrow{r}{f} \arrow{ur}{\widetilde{f}}
&
X
\end{tikzcd}
\]
is called the \emph{distributional category} of $f$, denoted $\dcat(f)$, see~\cite{Dr,Ja2}. From the proof of Lemma~\ref{lem:dsecat1} below, one can see that $\dcat(f)=\dsecat(f^*e^X)$, where $f^*e^X$ is the pullback of the evaluation fibration $e^X$ along $f$. In particular, we have that $\dcat(X)=\dcat(1_X)=\dsecat(e^X)$.
We note that this definition is equivalent to the original one in~\cite{DJ}, which fits better with the viewpoint of the Introduction.

If $E=P(X)$ is the free path space and $f=\text{ev}_m^X$ is the evaluation fibration, then the least $n$ such that $(\text{ev}^X_m)_{n+1}
$ has a section is called the $m$\emph{-th sequential distributional complexity} of $X$, denoted $\dTC_m(X)$, see~\cite{DJ,Ja1} (a similar invariant was denoted $\sf{{ATC}}$$_m(X)$ in~\cite{KW}). Therefore, $\dTC_m(X)=\dsecat(\text{ev}^X_m)$ for each $m\ge 2$.

In general, $\dsecat(p)\le\secat(p)$ for any fibration $p\colon E\to B$ whose base $B$ is paracompact, because there exists a (continuous) fiberwise map $\psi_k\colon \ast^k E\to E_k$ for each $k\ge 1$ --- see~\cite{KK,Ja1,KW2} for details about this map. Hence, in particular, $\dcat(X)\le\cat(X)$ and $\dTC_m(X)\le\TC_m(X)$ holds.

We now prove some standard properties of $\dsecat$ that will be useful later.
\begin{lemma}\label{lem:dsecat1}
    If $p:E\to B$ is a fibration and $f:A\to B$ is a map, then $\dsecat(f^*p)\le\dsecat(p)$, where $f^*p$ is the pullback of $p$ along $f$.
\end{lemma}

\begin{proof}
    Let $\dsecat(p)=n$ and $s:B\to E_{n+1}$ be a section of $p_{n+1}:E_{n+1}\to B$. For each $x\in A$, we have a homeomorphism between $(f^*p)^{-1}(x)$ and $p^{-1}(f(x))$.
    %The fiber of $x\in A$ under $f^*p$ is homeomorphic to the fiber of $f(x)\in B$ under $p$.
    Since the functor $\B_{n+1}$ preserves homeomorphisms, the fibers of the fibrations $p_{n+1}$ and $(f^*p)_{n+1}:(f^*E)_{n+1}\to A$ are homeomorphic and we have the homeomorphism
    \[
    f^*E_{n+1}=\bigcup_{x\in A} \B_{n+1}\left(p^{-1}(f(x))\right)=\bigcup_{x\in A} \B_{n+1}\left((f^*p)^{-1}(x)\right)=(f^*E)_{n+1}.
    \]
    Hence, we have the following (homotopy) pullback diagram:
    \[\begin{tikzcd}
     A\arrow[ddr, bend right, "1_B"] \arrow[drr, bend left, "s\circ f"] \arrow[dr, dashed, "\tau"] \\
 &
    (f^*E)_{n+1} \arrow{r} \arrow{d}{(f^*p)_{n+1}} & E_{n+1} \arrow[d, "p_{n+1}"] \\
    & A \arrow[swap]{r}{f}                  & B,
    \end{tikzcd}\]
    where $\tau:A\to (f^*E)_{n+1}$ exists by the universal property of the pullback. Thus, we get a section of $(f^*p)_{n+1}$ and hence the inequality $\dsecat(f^*p)\le n$.
\end{proof}

Suppose $p \colon E\to B$ is a covering with, as is our convention, $B$ a (connected and, of course, locally connected) CW complex. If $H=\pi_1(E)\subset\pi_1(B)$, then we can construct $E$ up to homeomorphism over $B$ from the based path space $P_0(B)$ as follows (see~\cite[Chapter~2, Section~5]{Sp}). For each $\gamma\in P_0(B)$, we say $\gamma\sim_H\alpha$ for some $\alpha\in P_0(B)$ if $\gamma(0)=\alpha(0)$ and $\bar\gamma\ast\alpha$ is a loop whose based homotopy class in $B$ represents an element of $H$. Then $E=P_0(B)/\sim_H$. The quotient map $Q:P_0(B)\to E$ satisfies $e^B=p\circ Q$. For each $n$, we get an induced map $Q_{n}:P_0(B)_n\to E_n$ satisfying $e^B_n=p_n\circ Q_n$.

\begin{lemma}\label{lem:dsecat2}
    If $p:E\to B$ is a covering and $f:A\to B$ is a map, then $\dsecat(f^*p)\le\dcat(f)$, where $f^*p$ is the pullback of $p$ along $f$.
\end{lemma}

\begin{proof}
    Suppose $\dcat(f)=n$ and $s:A\to P_0(B)_{n+1}$ is a lift of $f$ along $e^B_{n+1}$, i.e., $e^B_{n+1}\circ s=f$. Then $p_{n+1}\circ Q_{n+1}\circ s=e^B_{n+1}\circ s=f$. Thus proceeding as in Lemma~\ref{lem:dsecat1}, from the universal property of the pullback of $p_{n+1}$ along $f$, we get a section of $(f^*p)_{n+1}$ and hence the desired inequality $\dsecat(f^*p)\le n$.
\end{proof}

    \begin{comment}
 \[\begin{tikzcd}
 A'\arrow[ddr, bend right, "1_A"] \arrow[drr, bend left, "Q_{n+1}\circ s"] \arrow[dr, dashed, "\tau"] \\
 & (f^*B)_{n+1} \arrow{d}{(f^*p)_{n+1}} \arrow{r}{} & E_{n+1}(p) \arrow{d}{p_{n+1}} \\%
 &A \arrow{r}[swap]{f} & B,
\end{tikzcd}\]
we get a homotopy section of $(f^*p)_{n+1}$, which gives us an actual section to $(f^*p)_{n+1}$ and the desired inequality $\dsecat(f^*p)\le n$.
\end{comment}

We note that the inequality $\secat(f^*p)\le\min\{\secat(p),\cat(f)\}$ holds for any Hurewicz fibration $p:E\to B$ (see~\cite[Section 9.3]{CLOT} for the case when $f=1_B$).

\section{Distributional one-category}\label{sec:dcat1}
In this section, we define a new invariant of distributional type as an analogue to the classical invariant $\cat_1$, which we now describe. 

In the notations of Section~\ref{subsec:secat section}, if $E = \widetilde X$ is the universal cover of $X$ and $f$ is the covering map $q^X$, then the \emph{one-category} of $X$, denoted $\cat_1(X)$, is defined 
as $\cat_1(X):=\secat(q^X)$ --- we refer to~\cite{Op} for a survey on $\cat_1$ (see also~\cite{OS}). Similarly, for a map $f\colon Y \to X$, we may define $\cat_1(f) = \secat(f^*q^X)$. 

One particular result about $\cat_1$ that is often useful is the following.

\begin{theorem}\label{thm:cat1EG}
$\cat_1(X)$ is the least integer $k$ such that there is a map $X \to L$ with $\hdim(L)=k$ that induces an isomorphism of fundamental groups. 
\end{theorem}
For an Eilenberg--Mac~Lane space $X = K(\pi,1)$, the theorem can be shown to give 
$\cat_1(X)=\cd(\pi)$, where $\cd(\pi)$ is the cohomological dimension of $\pi$ (see~\cite{Op} for instance). 

\begin{definition}
    If $q\colon \Wi{X}\to X$ is the universal covering, then the least integer $n$ such that $q_{n+1}^X$ has a section is called the \emph{distributional one-category} of $X$, denoted $\dcat_1(X)$. More generally, if
$f\colon Y \to X$ is a map, then the least $n$ such that $f$ has a lift $\widetilde f$ in
\[
\begin{tikzcd}[contains/.style = {draw=none,"\in" description,sloped}]
&
\widetilde X_{n+1} \arrow{d}{q_{n+1}^X}
\\
Y  \arrow{r}{f} \arrow{ur}{\widetilde{f}}
&
X
\end{tikzcd}
\]
is called the \emph{distributional one-category} of $f$, denoted $\dcat_1(f)$.
\end{definition}

Again, it follows from the proof of Lemma~\ref{lem:dsecat1} that $\dcat_1(f)=\dsecat(f^*q^X)$, where $f^*$ is the pullback of the universal covering map $q^X:\widetilde{X}\to X$ along $f$. If $f=1_X$, then we write $\dcat_1(1_X)=\dcat_1(X)$ and so, $\dcat_1(X)=\dsecat(q^X)$.

\subsection{Basic properties}\label{subsec: basic properties}

If $X$ is not simply connected, then there cannot be a section of the universal covering map because the composition of the induced maps 
on fundamental groups --- which must be the identity --- factors through the trivial group. Hence, $\cat_1(X)=0$ if and only if $X$ is simply connected. The same 
argument says that $\dcat_1(X)=0$ if and only if $X$ is simply connected. Furthermore, we have the following.

\begin{proposition}\label{prop:min1cat1}
For any $f:Y\to X$, we have that
$$\dcat_1(f) \leq \cat_1(f).
$$
\end{proposition}
\begin{proof}
The space $\widetilde X_{n+1}$ is homotopic to a quotient of the $(n+1)$-th iterated fiber join $\ast^{n+1} \widetilde X$ by~\cite{KK} (see also~\cite[Section 8]{Ja1}), and $\cat_1(f)$ is the least $n$ such that there is a lift of $f:Y\to X$ along the fibration $\ast^{n+1} \widetilde X \to X$. Any such lift then induces a lift of $f$ along $q_{n+1}^X\colon\widetilde X_{n+1} \to X$.
\end{proof}

\begin{proposition}\label{prop:min}
For any $f\colon Y \to X$, we have
\[
\dcat_1(f) \leq \min\{\dcat(f),\dcat_1(Y),\, \dcat_1(X)\}.
\]
\end{proposition}
\begin{proof}
Since $\dcat_1(f)=\dsecat(f^*q^X)$, the inequality $\dcat_1(f)\le\dcat(f)$ follows from Lemma~\ref{lem:dsecat2} upon taking the covering map as the universal covering $q^X$.

Let $\dcat_1(Y)=n$. Suppose $\wt{f}\colon \wt Y\to \wt X$ is the obvious map covering $f$, and $\wt{f}_{n+1}\colon \wt{Y}_{n+1}\to \wt{X}_{n+1}$ is the natural map over $f$ induced by $\wt f$, see Remark~\ref{rem:naturality}.
A section $s\colon Y \to \widetilde Y_{n+1}$ gives a lifting $g =
\wt{f}_{n+1} \circ s$ with
\[
q_{n+1}^X \circ  g = q_{n+1}^X \circ \wt{f}_{n+1} \circ s = f\circ q_{n+1}^Y \circ s = f
\]
and so $\dcat_1(f) \leq n$.

If $\dcat_1(X)=n$, then a section $s\colon X \to \widetilde X_{n+1}$ gives $h = s\circ f$
with
$$q_{n+1}^X \circ h = q_{n+1}^X\circ s \circ f = f$$
and so $\dcat_1(f) \leq n$.
\end{proof}

The definition of the universal cover $\widetilde X$ as fixed endpoint homotopy classes of based paths in
$X$ gives a map $\omega\colon P_0(X) \to \widetilde X$ with $q^X \circ \omega = e^X$ (\emph{cf}. Section~\ref{sec:basic2}), which is a fiber homotopy
equivalence when $X$ is an Eilenberg--Mac~Lane space $K(\pi,1)$. The map $\omega$ induces a map $\omega_n\colon P_0(X)_n \to \widetilde X_n$ for all $n$
with $q_n^X \circ \omega_n = e_n^X$.

\begin{proposition}\label{prop:XKpi1}
For any $f\colon Y \to X$ with $X=K(\pi,1)$, we have
$$\dcat_1(f) = \dcat(f).$$
In particular, $\dcat_1(\pi)=\dcat(\pi)$.
\end{proposition}
\begin{proof}
In view of Proposition~\ref{prop:min}, we must only show that $\dcat_1(f) \geq \dcat(f)$. If $\dcat_1(f) = n$, then there is a lift $\widetilde f\colon
Y \to \widetilde X_{n+1}$ with $q_{n+1}^X\circ \widetilde f = f$. Now, $\omega_{n+1}\colon P_0(X)_{n+1}
\to \widetilde X_{n+1}$ is a homotopy equivalence over $X$, so it is a fiber homotopy equivalence as well. Suppose $\omega_{n+1}^{-1}:\widetilde X_{n+1}\to P_0(X)_{n+1}$ is a homotopy inverse of $\omega_{n+1}$. Then we have for the composition $\omega_{n+1}^{-1} \circ \widetilde f \colon Y \to P_0(X)_{n+1}$ that
\[
e_{n+1}^X \circ \omega_{n+1}^{-1} \circ \widetilde f \simeq q_{n+1}^X\circ\widetilde f = f.
\]
Now, the map $w_{n+1}^{-1} \circ \widetilde{f}$ helps us obtain a lift of $f$ (\emph{cf.} Remark~\ref{rmk:homotopy section}). This implies the inequality $\dcat(f) \leq n$.
\end{proof}

\begin{proposition}\label{prop:classmap}
Let $\kappa\colon X \to K(\pi,1)$ classify the universal cover. For any map $f\colon Y \to X$, we have
$$\dcat_1(f) \leq \dcat(\kappa\circ f) \leq \dcat_1(Y).$$
\end{proposition}
\begin{proof}
Up to homotopy, we have a pullback diagram
\[
\begin{tikzcd}[contains/.style = {draw=none,"\in" description,sloped}]
\widetilde X \arrow{r}{} \arrow[swap]{d}{q^X}
&
P_0(K) \arrow{d}{e^K}
\\
X  \arrow{r}{\kappa}
&
K,
\end{tikzcd}
\]
where $K = K(\pi,1)$. We now have the following homotopy pullback diagram (\emph{cf.} the proof of Lemma~\ref{lem:dsecat1})
\[
\begin{tikzcd}[contains/.style = {draw=none,"\in" description,sloped}]
\widetilde X_{n+1} \arrow{r}{} \arrow[swap]{d}{q_{n+1}^X}
&
P_0(K)_{n+1} \arrow{d}{e_{n+1}^K}
\\
X  \arrow{r}{\kappa}
&
K.
\end{tikzcd}
\]
If $\dcat_1(\kappa\circ f)=n$, we have a lift $\widetilde{\kappa f}\colon Y\to P_0(K)_{n+1}$ with $e_{n+1}^K\circ \widetilde{\kappa f}=\kappa\circ f$. Using the universal property of this pullback, we obtain a lift of $f$ along $q_{n+1}^X$, and therefore,
$\dcat_1(f) \leq \dcat_1(\kappa\circ f)$. By Proposition~\ref{prop:XKpi1}, we have the equality $\dcat_1(\kappa\circ f)
= \dcat(\kappa\circ f)$, so the first inequality is proven. The second inequality follows from Proposition
\ref{prop:min}.
\end{proof}

For the categorical invariant $\cat_1(Y)$, we have that $\cat_1(Y) = \cat(\kappa)$, where $\kappa\colon
Y \to K(\pi_1(Y),1)$ classifies the universal cover, see~\cite[Section 9.3]{CLOT}. Here is the analogous result for $\dcat_1$.

\begin{corollary}\label{cor:classmap}
Let $\kappa\colon Y \to K(\pi,1)$ classify the universal cover. Then
$$\dcat_1(Y) = \dcat_1(\kappa) = \dcat(\kappa).$$
\end{corollary}
\begin{proof}
In Proposition~\ref{prop:classmap}, let $f=1_Y$. We then get
$$\dcat_1(Y) = \dcat_1(1_Y) \leq \dcat(\kappa) \leq \dcat_1(Y)$$
using Proposition~\ref{prop:XKpi1} so that the conclusion follows.
\end{proof}

\begin{proposition}\label{prop:compmin}
For the composition $X \stackrel{f}{\to} Y \stackrel{g}{\to} Z$, we have
$$\dcat_1(g \circ f) \leq \min\{\dcat_1(f),\, \dcat_1(g)\}.$$
\end{proposition}
\begin{proof}
Suppose $\dcat_1(f)=n$. Then we have a diagram
\[
\begin{tikzcd}[contains/.style = {draw=none,"\in" description,sloped}]
&
\widetilde Y_{n+1} \arrow{r}{} \arrow[swap]{d}{}
&
\widetilde Z_{n+1} \arrow{d}{}
\\
X  \arrow{r}{f}  \arrow{ur}{\widetilde f}
&
Y \arrow{r}{g}
&
Z
\end{tikzcd}
\]
which shows that $\dcat_1(g\circ f) \leq n$.

Suppose $\dcat_1(g)=n$. Then we have a diagram
\[
\begin{tikzcd}[contains/.style = {draw=none,"\in" description,sloped}]
&
&
\widetilde Z_{n+1} \arrow{d}{}
\\
X  \arrow{r}{f}
&
Y \arrow{r}{g}  \arrow{ur}{\widetilde g}
&
Z
\end{tikzcd}
\]
which shows $\dcat(g\circ f) \leq n$.
\end{proof}

\begin{corollary}\label{cor:isoineq}
If $f\colon Y \to X$ induces an isomorphism of fundamental groups, then
$$\dcat_1(Y) \leq \dcat_1(X).$$
\end{corollary}
\begin{proof}
We have a diagram
$$\xymatrix{
Y \ar[rr]^-{\kappa_Y} \ar[dr]_-f && K(\pi,1) \\
& X, \ar[ur]_-{\kappa_X}
}
$$
where $\kappa_Y$ and $\kappa_X$ classify the respective universal covers. By Corollary
\ref{cor:classmap} we know that $\dcat_1(Y)=\dcat_1(\kappa_Y)$ and $\dcat_1(X) =
\dcat_1(\kappa_X)$. By Proposition~\ref{prop:compmin}, we then have
$$\dcat_1(Y) =\dcat_1(\kappa_Y) \leq \min\{\dcat_1(f),\, \dcat_1(X)\}\leq \dcat_1(X).$$
The above inequality also gives $\dcat_1(Y)\le\dcat_1(f)$, which, by Proposition~\ref{prop:min} implies $\dcat_1(f)=\dcat_1(Y)$.
\end{proof}

In particular, we get the following using Corollary~\ref{cor:isoineq} and Proposition~\ref{prop:min}.

\begin{corollary}\label{cor:oneside}
    If $s:X\to L$ induces an isomorphism of fundamental groups, then $\dcat_1(X)\le\dcat(L)$.
\end{corollary}
This corollary can also be obtained from Lemma~\ref{lem:dsecat1} and Proposition~\ref{prop:min}.
\begin{remark}
    The (weaker) converse of Corollary~\ref{cor:oneside} is not true. Indeed, we have $\dcat_1(\R P^n)\le\dcat(\R P^n)= 1$ for each $n$ by Proposition~\ref{prop:min} and~\cite{DJ,KW}, and if $\dim(L)=1$ for a non-contractible CW complex $L$, then $\cat(L)=1$ and $L$ is a co-$H$-space~(see~\cite{CLOT}), so that $\pi_1(L)$ is free. Hence, $s:\R P^n\to L$ cannot induce an isomorphism of fundamental groups for any $n\ge 2$. This should be compared with Theorem~\ref{thm:cat1EG}.
\end{remark}

\begin{proposition}\label{prop:retracts}
If $i\colon A \to X$ is a homotopy retract, then $\dcat_1(A) \leq \dcat_1(X)$.
\end{proposition}
\begin{proof}
Let $r\colon X \to A$ be the homotopy retraction with $r\circ i \simeq 1_A$. Let
$\dcat_1(X)=n$ with section $s\colon X \to \widetilde X_{n+1}$. We then have the following commutative diagram:
$$\xymatrix{
& \widetilde X_{n+1} \ar[d]^-{q^X_{n+1}} \ar[r]^-{r_{n+1}} & \widetilde A_{n+1}
\ar[d]^-{q^A_{n+1}} \\
A \ar[ur]^-{s\circ i} \ar[r]^-i & X \ar@/^0.2cm/[u]^-s \ar[r]^-r & A.
}
$$
Therefore,
\[
q^A_{n+1}\circ r_{n+1}\circ s\circ i = r\circ q_{n+1}^X\circ s\circ i= r\circ i \simeq 1_A.
\]
The map $r_{n+1}\circ s\circ i$ yields a section of $q^A_{n+1}$. Thus, $\dcat_1(A) \leq n$ holds.
\end{proof}

%Now, there are various independent proofs of the homotopy invariance of $\dcat_1$.
Now, the homotopy invariance of $\dcat_1$ can be justified.

\begin{corollary}\label{cor:homo inv}
If $f\colon Y \to X$ is a homotopy equivalence, then
\[
\dcat_1(Y)=\dcat_1(X).
\]
\end{corollary}
\begin{proof}
If $g\colon X \to Y$ is a homotopy inverse for $f$, then $f$ and $g$ are homotopy retracts of each other. So, the inequalities $\dcat_1(X)\le\dcat_1(Y)$ and $\dcat_1(Y)\le\dcat_1(X)$ follow from Proposition~\ref{prop:retracts}.
\end{proof}

Corollary~\ref{cor:homo inv} can also be proven using Corollary~\ref{cor:isoineq} instead.

\subsection{Behavior on covering maps}\label{subsec: behavior covering}
We now aim to generalize Corollary~\ref{cor:isoineq} from isomorphisms to monomorphisms. To do that, we first look at covering spaces. In this paper, the symbol $\simeq_*$ means homotopy rel endpoints for paths.

\begin{lemma}\label{lem:lifthomo}
Suppose $p \colon Y \to X$ is a covering and $\omega$, $\sigma$ are two paths in $Y$ with $\omega(0)
= \sigma(0)$ and $p\circ \omega \simeq_* p \circ \sigma$. Then $\omega \simeq_* \sigma$.
\end{lemma}
\begin{proof}
Let $H\colon I \times I\to X$ be a homotopy with $H_0=p\circ \omega$ and $H_1 = p\circ \sigma$ such that
$H(0,s)= p(\omega(0)) = p(\sigma(0))$ and $H(1,s) = p(\omega(1)) = p(\sigma(1))$ for each $s\in I$. Here, we use the convention $H_s(t)=H(t,s)$. We have a lift
$G$ in the commutative diagram
\[
\begin{tikzcd}[contains/.style = {draw=none,"\in" description,sloped}]
I \times \{0\} \arrow{r}{\omega} \arrow[swap]{d}
&
Y \arrow{d}{p}
\\
I\times I  \arrow{r}{H} \arrow[ur, dashed, "G"]
&
X.
\end{tikzcd}
\]
Therefore, $p\circ G=H$ with $G_0=\omega$ and $p\circ G_1=H_1=p\circ \sigma$. We thus have $p(G(0,s))=H(0,s)=p(\omega(0))
=p(\sigma(0))$ and $p(G(1,s))=H(1,s)=p(\omega(1))=p(\sigma(1))$. Hence, $G(0,s) \in p^{-1}(p(\omega(0)))$
and $G(1,s) \in p^{-1}(p(\omega(1)))$. These fibers are discrete and $G$ is continuous, so for all $s\in I$, we must have
$G(0,s)=\bar y$ and $G(1,s)=\widehat y$ for some points $\bar y,\widehat y\in Y$. Now, $G(0,0)=G_0(0)=\omega(0)$, so $\bar y=\omega(0)$. Also, $G(1,0)=G_0(1)=\omega(1)$, so $\widehat y=\omega (1)$. Note that $G_1(0)=G(0,1)=\bar y=\omega(0)=\sigma(0)$ and $p\circ G_1=p\circ \sigma$. By the unique path lifting property of $p$, we must have $G_1 = \sigma$. Therefore, $G$ gives
a homotopy $\omega =G_0\simeq_* G_1 = \sigma$. Observe that $\sigma(1)=G_1(1)=\widehat y=\omega(1)$, so $\omega$ and $\sigma$ have the same terminal points.
\end{proof}

\begin{remark}\label{rem:lifthomo2}
    It follows from the technique of the proof of Lemma~\ref{lem:lifthomo} that $\omega\simeq_{*}\sigma$ is true again if $\omega(1)=\sigma(1)$ and $p\circ \omega\simeq_{*}p\circ\sigma$.
\end{remark}

Given a covering map $p:Y\to X$, we use the definition of universal coverings to get
\[
\widetilde X = P_0(X)/[-]\ \text{ and }\ \widetilde Y = P_0(Y)/\{-\},
\]
where $[\alpha]=[\beta]$ for $\alpha,\beta\in P_0(X)$ if $\alpha(0)=\beta(0)$ and $\alpha\simeq_* \beta$, and $\{\sigma\}
=\{\tau\}$ for $\sigma,\tau \in P_0(Y)$ if $\sigma(0)=\tau(0)$ and $\sigma\simeq_* \tau$. We recall that $P_0(X)$ consists of paths in $X$ having terminal point $x_0$, and $P_0(Y)$ consists of paths in $Y$ having terminal point $y_0\in p^{-1}(x_0)$. The universal covering maps are given by $q^X([\gamma])=\gamma(0)$ and $q^Y(\{\alpha\})=\alpha(0)$. Of course, we know that $\widetilde X \cong \widetilde Y$, but we need the exact way this happens.

Define $P\colon \widetilde Y \to \widetilde X$ by $P(\{\sigma\}) = [p\circ\sigma]$. Clearly if
$\sigma\simeq_*\tau$, then $p\circ\sigma\simeq_* p\circ\tau$, so $P$ is well-defined.
Define $\theta\colon \widetilde X \to \widetilde Y$ as follows. Let $[\gamma]\in \widetilde X$. Because $\gamma(1)=x_0$, there is a unique lift $\widetilde\gamma\colon I \to Y$ to $\gamma$ ending at
$\widetilde\gamma(1)=y_0$. Define $\theta([\gamma])=\{\widetilde\gamma\}$. This map is well-defined
by Remark~\ref{rem:lifthomo2} since $\gamma \simeq_* \delta$ says that the unique 
lifts $\widetilde\gamma$
and $\widetilde \delta$ also have $\widetilde\gamma \simeq_* \widetilde\delta$ because $\widetilde\gamma(1) = y_0 =
\widetilde\delta(1)$. It is easy to check that
\[
P\circ\theta([\gamma])  = P(\{\widetilde\gamma\}) = [p\circ\widetilde\gamma]  = [\gamma] \ \text{ and } \
\theta\circ P(\{\sigma\}) = \theta([p\circ\sigma]) = \{\sigma\},
\]
where the last equality comes from the unique path lifting property of $p$. Now, recall that 
\[
\widetilde X_{n+1} = \{\mu \in \B_{n+1}(\widetilde X) \mid\, \text{supp}(\mu) \subset (q^X)^{-1}(x)\
\text{for\ some\ } x\in X\}, \ \text{ and}
\]
\[
\widetilde Y_{n+1} = \{\nu \in \B_{n+1}(\widetilde Y) \mid\, \text{supp}(\nu) \subset (q^Y)^{-1}(y)\
\text{for\ some\ } y\in Y\}.
\]
So, $\mu = \sum_i a_i [\gamma_i] \in \widetilde{X}_{n+1}$ means there exists $x\in X$ such that $q^X(\gamma_i)=\gamma_i(0)=x$ is independent of $i$, and $\nu=\sum_i b_i \{\alpha_i\} \in \widetilde{Y}_{n+1}$ means there exists $y\in Y$ such that $q^Y(\alpha_i)=\alpha_i(0)=y$ is independent of $i$.

\begin{comment}
Now, if we start with $\mu=\sum_i a_i [\gamma_i]$, then although each $\gamma_i$ starts at a fixed
$x$, the lifts $\widetilde\gamma_i$ of the $\gamma_i$ may start at different points in $p^{-1}(x)$. In order to
have them start at a fixed $y \in Y$ we fix paths $\omega_z\colon I \to Y$ with $\omega_z(0)=y$
and $\omega(1)=z$ for all $z \in p^{-1}(x)$. This must be done for all fibers of $p$ and all $y$ in a fiber. Let
$z_i = \widetilde\gamma_i(0)$ and note that $\omega_{z_i} * \widetilde\gamma_i$ is a path in $Y$ from
$y$ to $y_0$. That is, $\omega_{z_i} * \widetilde\gamma_i \in P_0(Y)$.
\end{comment}

\begin{theorem}\label{theo:mainresult1}
If $p\colon Y \to X$ is a covering map, then $\dcat_1(Y) \leq \dcat_1(X)$.
\end{theorem}
\begin{proof}
Suppose $x_0\in X$ is a basepoint of $X$. We choose $y_0\in p^{-1}(x_0)$ as a basepoint of $Y$. For each $z\in p^{-1}(x_0)$, we choose and fix a path $\omega_z\in P_0(Y)$ with the property that $\omega_z(0)=z$ and $\omega_z(1)=y_0$. Let $\dcat_1(X)=n$, and let $s\colon X \to \widetilde X_{n+1}$ be a section of $q_{n+1}^X:\widetilde{X}_{n+1}\to X$. We will now define a map $\tau\colon Y \to \widetilde Y_{n+1}$ as follows. Suppose
\[
s(x) = \sum_i a_i [\gamma_i]\ \text{ with }\ \gamma_i(0)=x\ \text{ for\ all}\  i.
\]
Let $y\in p^{-1}(x)$. Then $\gamma_i$ lifts to a unique path $\widetilde{\gamma}_i\in P(Y)$ such that $\widetilde{\gamma}_i(0)=y$. If $z_i=\widetilde{\gamma}_i(1)$, then $\widetilde{\gamma}_i\ast \omega_{z_i} \in P_0(Y)$. We define
\[
\tau(y) = \sum_i a_i \left\{\widetilde{\gamma}_i\ast\omega_{z_i}\right\}.
\]
Note that for each $i$, we have $(\widetilde{\gamma}_i\ast\omega_{z_i})(0)=\widetilde{\gamma}_i(0)=y$, so
$\text{supp}(\tau(y)) \subset (q^Y)^{-1}(y)$. Suppose $\gamma_i\simeq_* \alpha$. Then Lemma~\ref{lem:lifthomo} gives $\widetilde\gamma_i \simeq_* \widetilde\alpha$, which entails the equality $z_i=\widetilde\gamma_i(1)=\wt \alpha(1)$. Since we fixed the paths $\omega_{z}$ using only their endpoints $\{z,y_0\}$, the equality $z_i=\wt\alpha(1)$ implies $w_{z_i}=w_{\wt\alpha(1)}\in P_0(Y)$. So, we have
\[
\widetilde\gamma_i \ast \omega_{z_i} \simeq_* \wt\alpha \ast \omega_{z_i} \simeq_* \wt\alpha\ast \omega_{\wt\alpha(1)} .
\]
Therefore, $\tau$ is well-defined, and hence, it is a section of $q_{n+1}^Y:\widetilde{Y}_{n+1}\to Y$. Thus, we obtain $\dcat_1(Y) \leq n=\dcat_1(X)$.
\end{proof}

By Corollary~\ref{cor:classmap} and Theorem~\ref{theo:mainresult1}, for the covering $f\colon Y\to X$ and classifying maps $\kappa_X\colon X\to K(\pi_1(X),1)$ and $\kappa_Y\colon Y\to K(\pi_1(Y),1)$, we have the inequality $\dcat(\kappa_Y)=\dcat_1(\kappa_Y)\le\dcat_1(\kappa_X)=\dcat(\kappa_X)$. This is the analogue of the inequality $\cat_1(Y)=\cat(\kappa_Y)=\cat_1(\kappa_Y) \le \cat_1(\kappa_X)=\cat(\kappa_X)=\cat_1(X)$ from~\cite{OS}.

\begin{corollary}\label{cov:covineq}
If $f\colon Y \to X$ induces a monomorphism of fundamental groups, then
\[
\dcat_1(Y)\le\dcat_1(X).
\]
\end{corollary}
\begin{proof}
    Let $\pi= \Im(f_*\colon \pi_1(Y) \to \pi_1(X)) \subset \pi_1(X)$. We define a covering $p:\widehat{X}\to X$ with $\pi_1(\widehat{X})=\pi$. By the lifting criterion, we get a lift $\widehat{f}:Y\to \widehat{X}$ of $f$ along $p$. Note that $\widehat{f}$ induces an isomorphism of fundamental groups. Hence,
    \[
    \dcat_1(Y)\le\dcat_1(\widehat{X})\le\dcat_1(X)
    \]
    is obtained using Corollary~\ref{cor:isoineq} and Theorem~\ref{theo:mainresult1}, respectively.
\end{proof}

\section{Distributional one-category and the fundamental group}\label{sec: relationship with pi1}
We begin by noting that for any Eilenberg--Mac~Lane space $X=K(\pi,1)$, we immediately have $\dcat_1(X)\le \cat_1(X)=\cat(X)=\cd(\pi)$. Also, Corollary~\ref{cor:classmap} gives $\dcat_1(X)=\dcat(X)=\dcat(\pi)$. Moreover, if $\cat(X)=\cd(\pi)$ is finite, then $\pi$ is torsion-free and so, $\dcat(\pi)=\cd(\pi)$ holds due to~\cite{KW} (see Section~\ref{subsec: 
torsion-free} for a proof from our viewpoint), which means $\dcat_1(X)=\cd(\pi)$.

In the case when $X$ is \emph{not} an Eilenberg--Mac~Lane space and its universal cover $\widetilde{X}$ is highly connected (but \emph{not} contractible), we now obtain a similar algebraic description of $\dcat_1(X)$ and $\cat_1(X)$ under an additional assumption on $\cd(\pi_1(X))$. The following statement can be seen as a partial extension of the above result for Eilenberg--Mac~Lane spaces.

\begin{theorem}\label{theo:mainresult2}
For $X$ with $\pi_1(X)=\pi$, suppose that $\widetilde X$ is $(k-1)$-connected
and that $\cd(\pi) \leq k$. Then
$$\dcat_1(X) = \cat_1(X)=\cd(\pi).$$
\end{theorem}
\begin{proof}
The Eilenberg--Mac~Lane space $K=K(\pi,1)$ may be constructed by adding cells
to $X$ of dimension greater than or equal to $k+1$. Therefore, we have
$\pi_j(K,X)=0$ for $j \leq k$. If we convert $X \hookrightarrow K$ to a
fibration with fiber $F$, then $\pi_s(F) = \pi_{s+1}(K,X)$, so that $\pi_s(F)=0$
for $s \leq k-1$. Since the fiber is $(k-1)$-connected, all possible obstructions to a section of $X \to K$ live in $H^{t+1}(K;\pi_t(F))=H^{t+1}(\pi;\pi_t(F))$ for $t\geq k$. In particular, the primary obstruction lives in $H^{k+1}(\pi;\pi_k(F))$. But $\cd(\pi)\le k$ implies $H^{t+1}(\pi;\pi_t(F))=0$ for all $t\ge k$ and hence, a section $K \to X$ exists. This makes $K$ into a homotopy retract of $X$. Then, we have the following inequality:
\[
\cd(\pi)=\dcat(K)=\dcat_1(K)\le\dcat_1(X)\le\dcat_1(K)\le\cd(\pi).
\]
The first equality is true because $\pi$ is torsion-free, the second one comes from Corollary~\ref{cor:classmap}, the third inequality follows from Proposition~\ref{prop:retracts}, the fourth one is by Corollary~\ref{cor:isoineq}, and the last one is standard for Eilenberg--Mac~Lane spaces. Analogous results similarly give $\cd(\pi)=\dcat_1(X)\le\cat_1(X)\le\cd(\pi)$.
\end{proof}

\begin{remark}\label{rmk:dcat vs dcat-pi}
    For any $X$ with $\pi_1(X)=\pi$, Corollary~\ref{cor:classmap} implies the inequality $\dcat_1(X)\le\dcat_1(\pi)\le \cd(\pi)$. Therefore, if $X$ and $\pi$ satisfy the hypotheses of Theorem~\ref{theo:mainresult2}, then we have $\dcat_1(X)=\dcat_1(\pi)$, so the first inequality is sharp. 
    
    On the other hand, suppose $X$ is such that $\dim(X)<\cd(\pi)$ and $\cd(\pi)$ 
is finite (i.e., $\pi$ is torsion-free), then we have
    \[
    \dcat_1(X) \le \dcat(X) \le \cat(X)\le \dim(X) < \cd(\pi) = \dcat(\pi)= \dcat_1(\pi).
    \]
    Here, $\cat(X)\le\dim(X)$ is well-known (see, for example,~\cite[Theorem 1.7]{CLOT}).
    Hence, $\dcat_1(X)<\dcat_1(\pi)$ is possible, and since every finitely presentable torsion-free discrete group $\pi$ can be realized as the fundamental group of a closed $4$-manifold $X$, the gap between $\dcat_1(X)$ and $\dcat_1(\pi)$ can be  made arbitrarily large. For an explicit example, take $X=SP^n(M_g)$, the $n$-th symmetric product of a closed orientable surface $M_g$ of genus $g>n\ge 2$. Then we have $\pi=\Z^{2g}$, so that $\dim(X)=2n<2g=\cd(\pi)$. To see an arbitrarily large gap for closed $4$-manifolds, take $X=SP^2(M_{k+2})$ and vary the integers $k\ge 1$.
\end{remark}

Since $\widetilde{X}$ is always simply connected, taking $k=2$ in Theorem~\ref{theo:mainresult2} implies the following.

\begin{corollary}\label{cor:cd1}
For $X$ with $\pi_1(X)=\pi$, suppose $\cd(\pi)\le 2$. Then we have that $\dcat_1(X)=\cat_1(X)=\cd(\pi)$.
\end{corollary}

So, the strict inequalities $\dcat_1(X)<\cat_1(X)$ and $\cat_1(X)<\cd(\pi)$ are possible only when $\cd(\pi)>2$.

\begin{ex}\label{ex:smallcd}
    If $\pi \ne \{1\}$ is free, then $\cd(\pi)=1$ (see~\cite[Example I.4.3]{Br}) and hence, $\dcat_1(X)=\cat_1(X)=1$ by 
    Corollary~\ref{cor:cd1}. Similarly, if the fundamental group of $X$ is a surface group, then $\dcat_1(X)=\cat_1(X)=2$.
\end{ex}

\begin{ex}\label{exam:onerelator}
If $\pi = \langle a_1,\ldots, a_n\mid r\rangle$ is a non-free group with a single relator $r$ that is not a power of another element (a so-called \emph{proper power}), then it is known by Lyndon's Identity Theorem~\cite{Ly} that $\cd(\pi)=2$. Therefore, any $X$ with 
such a fundamental group has $\dcat_1(X)=2$ by Corollary~\ref{cor:cd1}. Now, a cohomological lower bound for $\dcat$ is given by rational cup-length due to~\cite[Section 4.2]{DJ}, 
so if we take $Y=X \times \C P^n$ say, then $\dcat_1(Y)=2$ while $\dcat(Y) \geq n$. Therefore, the gap between $\dcat_1$ and $\dcat$ can
be arbitrarily large (\emph{cf.} Proposition~\ref{prop:min}).
\end{ex}

We now continue with more corollaries of Theorem~\ref{theo:mainresult2}.

\begin{corollary}\label{cor:cd2}
For $X$ with $\pi_1(X)=\pi$, suppose that $\widetilde X$ is $(k-1)$-connected and that $\cd(\pi) \leq k$. If $Y$ is a finite cover of $X$, then
\[
\dcat_1(Y)=\dcat_1(X)=\cd(\pi)=\cat_1(X)=\cat_1(Y).
\]
\end{corollary}

\begin{proof}
Note that $\pi_1(Y)$ is a finite index subgroup of $\pi$, which is torsion-free. Hence, Serre's theorem (see~\cite[Theorem VIII.3.1]{Br}) implies that $\cd(\pi_1(Y))=\cd(\pi)\le k$. We now apply Theorem~\ref{theo:mainresult2} to $Y$ and $X$ to get the desired equality.
\end{proof}

In particular, we have several non-trivial examples of finite coverings for which the inequality in Theorem~\ref{theo:mainresult1} is saturated.

In the classical setting, $\cat_1(X\times Z) \le \cat_1(X)+\cat_1(Z)$ holds for all $X$ and $Z$, yielding $\cat_1(X\times Z)=\cat_1(X)$ for simply connected spaces $Z$. We now show an analogous phenomenon for $\dcat_1$ under additional assumptions on $X$ and $Z$. This is interesting because in general, the inequality $\dcat_1(X\times Z) \le \dcat_1(X)+\dcat_1(Z)$ is \emph{not} true, even if $X=Z$, as we will show in Example~\ref{exam:pi1Z2}.

\begin{corollary}\label{cor:cd3}
    For $X$ with $\pi_1(X)=\pi$, suppose that $\widetilde X$ is $(k-1)$-connected for $k\ge 2$ and that $\cd(\pi) \leq k$. If $Z$ is any $(k-1)$-connected space, then
    \[
    \dcat_1(X\times Z)=\cat_1(X\times Z)=\cd(\pi)=\cat_1(X)=\dcat_1(X).
    \]
\end{corollary}
\begin{proof}
    Since $Z$ is simply connected, $\pi_1(X\times Z)=\pi$ and $\widetilde{X\times Z}=\widetilde{X}\times Z$. The latter space is $(k-1)$-connected by hypothesis. Therefore, Theorem~\ref{theo:mainresult2} applied to $X\times Z$ and $X$ gives the desired equality.
\end{proof}

In particular, when $\cd(\pi)\le 2$, we have $\dcat_1(X\times S^n)=\cd(\pi)=\dcat_1(X)$ for all $n\ge 2$.

We now see another rigidity result for distributional one-category (also see Corollary~\ref{cor:kunivcover}).

\begin{proposition}\label{prop: trivial}
    If $X$ has $\pi_1(X)=\pi$ finite, then $\dcat_1(X)\le |\pi|-1$.
\end{proposition}

\begin{proof}
    Due to~\cite[Theorem 7.2]{KW}, we have $\dcat(K(\pi,1))\le |\pi|-1$. Because $X\to K(\pi,1)$ is an isomorphism on fundamental groups, Corollaries~\ref{cor:classmap} and~\ref{cor:isoineq} imply
    \[
    \dcat_1(X)\le \dcat_1(K(\pi,1))=\dcat(K(\pi,1))\le |\pi|-1.
    \]
\end{proof}
In Corollary~\ref{cor:Kpifinite}, without using the result on $\dcat(K(\pi,1))$, we will give a direct proof of the inequality $\dcat_1(K(\pi,1)) \leq |\pi| - 1$. 

Taking $\pi=\Z_2$ in Proposition~\ref{prop: trivial}, we obtain the following.
%Corollary~\ref{cor:cd3} has some interesting consequences on the behavior of $\cat_1$ and $\dcat_1$ of products involving spheres. Before we explore them, we note the following.

\begin{corollary}\label{cor:dcat1=1}
If $X$ has $\pi_1(X)=\Z_2$, then $\dcat_1(X)=1$.
\end{corollary}

\begin{comment}
    For $K=K(\Z_2,1)$, the classifying map $\kappa:X\to K$ induces an isomorphism of fundamental groups. By Corollaries ~\ref{cor:classmap} and~\ref{cor:isoineq}, we have $\dcat_1(X)\le\dcat_1(K)=\dcat(K)=1$. But $X$ is not simply connected, so $\dcat_1(X)=1$.
\end{comment}

In stark contrast, $\cat_1(X)$ need not have a constant value if $\pi_1(X)=\Z_2$. For example, $\cat_1(\R P^\infty)=\cd(\Z_2)$ is infinite but $\cat_1(\R P^n)=n$ for each $n$ since $\R P^n$ is an essential $n$-manifold (see Section~\ref{sec:essential}). In light of this, we conclude that, surprisingly, the fundamental group of a space has a much stronger influence on its $\dcat_1$ value than it has on its $\cat_1$ value.

\begin{remark}
    For any simply connected space $Y$, $\dcat_1(X\times Y)=1=\dcat_1(X)$ if $\pi_1(X)=\Z_2$. Since $Y$ is a retract of $X\times Y$, we also have $\dcat(X\times Y)\ge \dcat(Y)$. Taking $Y=\C P^n$ for any $n$, we conclude that $\dcat(X\times \C P^n)\ge \dcat(\C P^n)=n$, see~\cite[Section 6.1]{DJ}. Hence, we have another example showing that the gap between $\dcat_1(Z)$ and $\dcat(Z)$ can be arbitrarily large (\emph{cf.} Example~\ref{exam:onerelator}).
\end{remark}

\section{Distributional (one-)category of discrete groups}\label{dcat1 of groups}

In this section, we study $\dcat_1(\pi)$ for discrete groups $\pi$. Using our new invariant, we give alternative (and, we hope, illuminating) proofs of the two key results of~\cite{KW} on the distributional category of discrete groups: a universal upper bound for $\dcat$ of finite groups, and the equivalence of the $\dcat$ of torsion-free groups with their cohomological dimension.

\subsection{The case of finite fundamental groups}\label{sec:finitefund}
We begin this section by noting a lower bound for the distributional (one-)category of finite groups. If $\pi$ is a finite group and $P$ is any $p$-subgroup\hspace{0.3mm}\footnote{\hspace{0.3mm}We recall that a finite group $\pi$ is called a $p$-group for a prime $p$ if $|\pi|=p^s$ for some $s\ge 1$.} of $\pi$, then the inequality
\[
\dcat(\pi)=\mathsf{acat}(\pi)\ge |P|-1
\]
was obtained in~\cite[Proposition 4.3]{KW2}. The original proof is difficult since it also gives a better bound $\mathsf{acat}(\pi)\ge 2|P|-1$ for non-self-normalizing $p$-subgroups $P$. For any group $\pi$, we have $\dcat_1(\pi)=\dcat(\pi)$ by Corollary~\ref{cor:classmap}. So, we have $\dcat_1(\pi)\ge |P|-1$. 
%due to~\cite{KW2}. 
Moreover, the upper bound 
$$\dcat(\pi) \leq |\pi| - 1$$
was also proved in~\cite{KW}, and we shall extend this result in Corollary 
\ref{cor:kunivcover} below.

In~\cite[Proposition 6.6]{KW}, an upper bound was given for the analog category of a space $X$ in terms of that of a finite
covering space $Y$. This formula also holds in the case of $\dcat_1$, and the proof is simplified (and
formulaic) due to the fact that our fibers involve probability measures on finite spaces. 

\begin{proposition}\label{prop:kcover}
Suppose $f\colon Y \to X$ is a finite cover with $k$-sheets. Then
$$\dcat_1(X) \leq k(\dcat_1(Y)+1) - 1.$$
\end{proposition}

\begin{proof}
Let $\dcat_1(Y)=n$ so that there is a section $s\colon Y \to \widetilde Y_{n+1}$
of the fibration $q^Y_{n+1} \colon \widetilde Y_{n+1} \to Y$, where $q^Y\colon
\widetilde Y \to Y$ is the universal covering. 
%The section $s$ has the form $$s(y) = \sum_{i=1}^{n+1} a_i z_i,$$ where $q^Y(z_i)=y$ for each $i$. 
Suppose $P\colon \Wi{Y}\to \Wi{X}$ is the homeomorphism induced by $f$ on the universal covers (\emph{cf.} Section~\ref{subsec: behavior covering}). Then we have the following diagram:
\begin{equation}\label{eq: diag}
\xymatrixcolsep{12mm}
\xymatrixrowsep{12mm}
  \xymatrix{
\widetilde Y_{n+1} \ar[r]^-{P_{n+1}} \ar[d]^-{q^Y_{n+1}} &
\widetilde X_{n+1} \ar[r]^-\varepsilon \ar[d]^-{q^X_{n+1}} &
\widetilde X_{k(n+1)} \ar[dl]^-{q^X_{k(n+1)}} \\
Y \ar[r]^-f \ar@/^1pc/[u]^-s & X,
}  
\end{equation}
where $\varepsilon$ is the inclusion. Let $f^{-1}(x) = \{y_1,\ldots, y_k\}$, and denote the image of the section $s$ on each $y_j$ by
\[
s(y_j) = \sum_{r=1}^{n+1} a_{rj} z_{rj},
\]
where $q^Y(z_{rj})=y_j$ for all $1\le r\le n+1$. We define a map $\tau\colon X\to \Wi{X}_{k(n+1)}$ by 
\begin{align*}
\tau(x) = 
&
\frac{1}{k} \sum_{j=1}^k (\varepsilon \circ P_{n+1}\circ s)(y_j) 
\\
= 
&
\frac{1}{k} \sum_{j=1}^k \varepsilon \left(P_{n+1}\left(\sum_{r=1}^{n+1} a_{rj} z_{rj}\right)\right)
\\
=
&
\frac{1}{k} \sum_{j=1}^k \sum_{r=1}^{n+1} a_{rj} P(z_{rj}),
\end{align*}
where the last equality holds since $\varepsilon$ is the inclusion. Now, we have
\[
q^X_{k(n+1)}(\tau(x))=q^X \circ P(z_{rj})=f\circ q^Y(z_{rj})=f(y_j)=x.
\]
Furthermore, note that
\[
\frac{1}{k} \sum_{j=1}^k \sum_{r=1}^{n+1} a_{rj}
= \frac{1}{k} \sum_{j=1}^k (1) = 1,
\]
so $\tau(x)\in \widetilde X_{k(n+1)}$ indeed. Hence, $\tau\colon X\to \Wi{X}_{k(n+1)}$ is a section of the fibration $q^X_{k(n+1)}\colon \Wi{X}_{k(n+1)}\to X$. Therefore, $\dcat_1(X) \leq k(n+1) - 1$.
\end{proof}

We then obtain an interesting bound for \emph{any} space $X$ where $\pi_1(X)$ is 
a finite group. We record this as follows.

\begin{corollary}\label{cor:kunivcover}
If $X$ has $\pi_1(X)=\pi$ finite with universal cover $q\colon \widetilde X \to X$, then
\[
\dcat_1(X) \leq |\pi| - 1,
\]
and the section $\tau\colon X \to \widetilde X_{|\pi|}$ is given by
\[
\tau(x) = \frac{1}{k} \Wi x_1 + \cdots + \frac{1}{k} \Wi x_{|\pi|},
\]
where $\{\Wi x_1, \ldots,\Wi x_{|\pi|}\} = q^{-1}(x)$.
\end{corollary}

\begin{proof}
We know that in general, $\dcat_1(Y)=0$ if and only if $Y$ is simply connected. So, applying Proposition~\ref{prop:kcover} to the universal covering $\Wi{X}$ yields $\dcat_1(X) \leq |\pi| - 1$.

Now, we only need to see the form of the section $\tau$. But because $\dcat_1(\widetilde X)=0$, the diagram in~\eqref{eq: diag} reduces to
$$
\xymatrixcolsep{15mm}
\xymatrixrowsep{12mm}
\xymatrix{
\widetilde X \ar[r]^-{P_1=1_{\widetilde X}} \ar[d]^-{1_{\widetilde X}} &
\widetilde X \ar[r]^-\epsilon \ar[d]^-{q^X} &
\widetilde X_{k} \ar[dl]^-{q^X_{k}} \\
\widetilde X \ar[r]^-{q^X} \ar@/^1pc/[u]^-{s=1_{\widetilde X}} & X.
}
$$
Then, for the points $\Wi x_j$ in the fiber of $x$, we have $s(\Wi x_j)
= \Wi x_j$. Since $\epsilon$ is the inclusion, this entails
$$\tau(x) = \frac{1}{k} \sum_{j=1}^k \Wi x_j.$$
\end{proof}

A special case is when $X = K(\pi,1)$ and $\pi$ is finite. Then $\widetilde X$
may be taken to be the contractible free $\pi$-space $E\pi$ and we may apply
Corollary~\ref{cor:kunivcover} to augment a result of~\cite{KW}. 

\begin{corollary}\label{cor:Kpifinite}
If $\pi$ is a finite group, then
$$\dcat(\pi) = \dcat_1(\pi) = \dcat_1(K(\pi,1)) \leq |\pi| - 1,$$
and the section $\tau \colon K(\pi,1) \to E\pi_{|\pi|}$ is obtained by taking the barycenter of the points in the fiber of the universal covering.
\end{corollary}
\begin{remark}\label{new addition}
 In~\cite{KW2} (also see~\cite{Dr}), it is shown that if $\pi$ is a finite $p$-group, then in fact
$\dcat_1(\pi)=\dcat(\pi)=|\pi| - 1$, so this provides other examples where $\dcat_1$ is 
non-trivial; indeed, it can be as large as desired. 
\end{remark}

\subsection{The case of torsion-free groups}\label{subsec: torsion-free}
In~\cite{KW}, it was shown that for a torsion-free discrete group $\pi$, a version
of the Eilenberg--Ganea theorem holds: namely (in our notation),
\[
\dcat(\pi) = \cd(\pi).
\]
This result has been instrumental in furthering research on distributional category and related invariants, see, for instance,~\cite{Dr,Ja2,JO} and Sections~\ref{sec: relationship with pi1} and~\ref{sec:essential}. A slightly different proof of this theorem was obtained in~\cite[Section 5.1]{Dr}. Here, we want to recast this result in terms of 
distributional one-category by providing an alternative proof using the classical 
invariant $\cat_1$. We don't claim that this viewpoint is substantially different from
the proof in~\cite{KW}, but there is a subtle discrete flavor to this proof that, we feel, makes it simpler.

Before we proceed, we recall that $\dcat_1(X) \leq \cat_1(X)$, where $\cat_1(X)$ is the least integer $k$ such that there is a map $X \to L$ with $\hdim(L)=k$ that induces an isomorphism of fundamental groups (see Theorem~\ref{thm:cat1EG}). The theorem of Eilenberg and Ganea gives $\cat_1(\pi) = \cd(\pi)$. We have also seen in Corollary~\ref{cor:classmap} that for a discrete group $\pi$, $\dcat_1(\pi) = \dcat(\pi)$, so that $\dcat_1(\pi)\le\cd(\pi)$ is always true. 

The exact same proof as in~\cite[Section 7]{KW} shows that $\pi$ acts freely on $\B_n(\pi)$ (by multiplication in $\pi$) if and only if $\pi$ is torsion-free. In fact, $\B_n(\pi)$ can be
identified with the $(n-1)$-skeleton of the infinite simplex $\Delta^\pi$, whose vertices are the elements of $\pi$ and where the $\pi$-action is simplicial. In particular, this means that the $\pi$-action is a covering space action in the terminology of~\cite{Ha}. Moreover, the covering is regular or normal (see~\cite[Proposition~1.40]{Ha}). 

Also, note that the measure space $\B_n(\pi)$ is simply connected when $n \ge 3$, and 
$\B_n(\pi) \simeq \bigvee S^{n-1}$ by~\cite{KK} for all $n$.

\begin{theorem}[\cite{KW}]\label{th: after kw}
If $\pi$ is torsion-free, then $\dcat(\pi) = \cd(\pi)$.
\end{theorem}

\begin{proof}
Since we know that $\dcat(\pi)=\dcat_1(\pi) \leq \cd(\pi)$, we must only prove the inequality
$\dcat_1(\pi) \geq \cd(\pi)$. 
Let $\dcat_1(\pi)=n$. If $n=0$, then $\pi=\{1\}$ by definition and so $\cd(\pi)=0$ as well. Therefore, we focus on the case $n\ge 1$. 

The universal covering of $B\pi=K(\pi,1)$ may be taken to be $q\colon E\pi 
\to B\pi$, where $\pi$ acts freely on the contractible space $E\pi$. Then, $n$ is the smallest integer for which the induced fibration $q_{n+1}\colon E\pi_{n+1} \to B\pi$ admits a section. Now, $q$ is a principal $\pi$-bundle map, and it is straightforward to see (as in 
\cite{KW}) that $q_{n+1}\colon E\pi_{n+1} \to B\pi$ is a fiber bundle with principal bundle $q$ and fiber $\B_{n+1}(\pi)$ with the given free action of $\pi$. The sections of $q_{n+1}$ are in bijective correspondence with equivariant maps $E\pi \to \B_{n+1}(\pi)$ --- see, for example,~\cite[Theorem~8.1]{Hu}. Since $\dcat_1(\pi)=n$, there is then such an equivariant map $\Phi$ that descends to 
a map on quotients
\[
\phi \colon B\pi = E\pi/\pi \to \B_{n+1}(\pi)/\pi.
\]
First, suppose $\dcat_1(\pi) = n \geq 2$. 
Then, since $\B_{n+1}(\pi)$ is simply connected, the map $\phi$ 
induces an isomorphism of fundamental groups (which are both $\pi$, of course). 
But the homotopy dimension of both $\B_{n+1}(\pi)\simeq \bigvee S^n$ and 
$\B_{n+1}(\pi)/\pi$ is $n$, so the definition of $\cat_1(\pi)$ and Theorem~\ref{thm:cat1EG} then 
say that 
\[
\cd(\pi) = \cat_1(\pi) \leq n = \dcat_1(\pi).
\]
For $n=1$, we no longer have that $\B_{n+1}(\pi)=\B_2(\pi)$ is simply connected, so we must argue 
a bit differently. Since $\B_2(\pi) \to \B_2(\pi)/\pi$ is a regular covering, it is also a principal $\pi$-bundle and thus, we have a map of principal $\pi$-bundles\hspace{0.3mm}\footnote{\hspace{0.3mm}Note that this is \emph{not} a morphism of principal bundles because the bases are different.}
$$\xymatrix{
E\pi \ar[r]^-\Phi \ar[d]_-q & \B_2(\pi) \ar[d] \\
B\pi \ar[r]^-\phi & \B_2(\pi)/\pi,
}
$$
which induces an isomorphism on fibers (because the actions are free and $\Phi$ is equivariant). 
We therefore have a homotopy commutative diagram
$$\xymatrix{
B\pi \ar[r]^-\phi \ar[d]_-\kappa & \B_2(\pi)/\pi \ar[d]^-{\kappa'} \\
K(\pi,1) \ar[r]^-\lambda_\cong & K(\pi,1),
}
$$
where $\kappa$ and $\kappa'$ are the classifying maps of the respective principal bundles. 
But $\kappa=\rm{id}_{B\pi}$ and $\lambda$ is an isomorphism, so the induced map on 
fundamental groups $\phi_\#\colon \pi \to \pi_1(\B_2(\pi)/\pi)$ must be injective. Then by 
\cite[Theorem~3.9]{OS}, we have that $\cat_1(B\pi) \leq \cat_1(\B_2(\pi)/\pi)$, which entails
$$\cd(\pi) = \cat_1(\pi) \leq \cat_1(\B_2(\pi)/\pi) \leq 1 = \dcat_1(\pi)$$
since the homotopy dimension of both $\B_2(\pi)$ and $\B_2(\pi)/\pi$ is equal to one. 
This completes the proof.
\end{proof}

\section{Distributional one-category of essential manifolds}\label{sec:essential}

Let $M$ be a closed $n$-manifold and $\pi_1(M)=\pi$. We say that $M$ is \emph{essential} in the sense of~\cite{Gr} if there exists a classifying map $\kappa \colon M \to K=K(\pi,1)$ that cannot be deformed into the $(n-1)$-skeleton of $K$. It is well-known that $M$ is essential if and only if $\cat_1(M)=\cat(M)=\dim(M)=n$ (see below). In this section, we aim to study the distributional one-category of some closed essential manifolds (in dimensions $\ge 3$) that have torsion-free fundamental group.

\begin{remark}
    We confine our discussion in this section to dimensions $n\ge 3$ since we will be using the theory of primary obstructions. We already know $\dcat_1(M)$ values when $M$ is essential and $n\in\{1,2\}$. Indeed, $\dcat_1(S^1)=1$, $\dcat_1(\R P^2)=1$, and $\dcat_1(M)=2$ for all closed surfaces $M\not\cong \R P^2$ by Example~\ref{ex:smallcd} and Theorem~\ref{theo:mainresult1}.
\end{remark}

Using obstruction theory (see, for example,~\cite[Section 8]{Ba}), it can be shown that $M$ is essential if and only if there exists a coefficient system $\A$ such that the map $\kappa^*\colon H^n(K;\A) \to H^n(M;\kappa^*(\A))$ is non-trivial for a classifying map $\kappa$. Now, consider the \emph{Berstein--Schwarz} class $\beta_\pi$ of $\pi$, which is the first obstruction to a lift of $K$ to its universal cover $\Wi K$,~\cite{Sch}. If $I(\pi)$ denotes the augmentation ideal of $\Z[\pi]$, then $\beta_\pi\in H^1(K;I(\pi))$. A theorem of Schwarz and Berstein says that $\cat(M)=n$ if and only if $\kappa^*(\beta_\pi^n)\ne 0$ in $H^n(M;\kappa^*(I(\pi)^{\otimes n}))$, see~\cite[Theorem 2.51]{CLOT} and~\cite{KR}. Hence, by taking $\A=I(\pi)$ above, we see that $\cat(M)=n$ if and only if $M$ is essential.

Recall from~\cite{Op} that $\cat_1(M)\le k$ if and only if there is a map $s\colon M \to L$ with $\hdim(L)=k$ such that $s$ induces an isomorphism of fundamental groups and the following diagram commutes:
\[
\begin{tikzcd}
    M \arrow{rr}{\kappa} \arrow{dr}{s} 
    &
    &
    K 
    \\
    &
    L. \arrow{ur}
    &
\end{tikzcd}
\]
Hence, $\cat_1(M)\ge r$ if $\kappa^*\colon H^r(K,\A) \to H^r(M;\kappa^*(\A))$ is non-zero for some local coefficient system $\A$. When $M$ is essential, we have $\kappa^*(\beta_\pi^n)\ne 0$, so that we can take $r=n$ and $\A=I(\pi)$ to get $\cat_1(M)\ge n$. This gives $\cat_1(M)=\cat(M)=n$.

\begin{remark}
  In the case of $\dcat_1$, the equality $\dcat_1(M)=\dim(M)$ is not true for essential $n$-manifolds $M$ in general. To see this, take an essential manifold $M$ with $\pi_1(M)$ finite and $\dim(M)\ge |\pi_1(M)|$ (for instance, take $M\simeq \R P^n$). Then $\dcat_1(M)<\dim(M)$ in view of Corollary~\ref{cor:kunivcover}. In particular, $\cat_1(M)$ and $\dcat_1(M)$ can be arbitrarily far apart for essential manifolds $M$. 
\end{remark}

In light of the above remark, we now work in the domain of essential manifolds having torsion-free fundamental groups (as in~\cite{Ja2}) to check which of these manifolds have maximum distributional one-category. 

\subsection{Essential manifolds having the cap property}\label{subsec: ess + cap}
If $M$ is an essential $n$-manifold with $\pi=\pi_1(M)$, then $\beta_\pi^n\ne 0$, so that $\cd(\pi)\ge n$ by definition. If $\pi$ is torsion-free, then an analogue of the Eilenberg--Ganea theorem (see Theorem~\ref{th: after kw}) implies that $\dcat(K)=\dcat(\pi)=\cd(\pi)\ge n$. But this means the $n$-th fibration $e^K_n:P_0(K)_n\to K$ cannot admit a section. Since the fiber $\B_n(\Omega K)$ of $e^K_n$ is $(n-2)$-connected~\cite{Dr} and $n\ge 3$, the primary obstruction to a section of $e^K_n$ is a class $\mathfrak{a}_\pi \in H^n(K;\F)$, where $\F$ is a $\pi$-module. Therefore, if $\hdim(K)=n$, then $\mathfrak{a}_\pi\ne 0$. 

Let us recall a homological condition on essential manifolds that guarantees $\mathfrak{a}_\pi\ne 0$ regardless of the assumption on $\hdim(K)$.

\begin{definition}[\protect{\cite{Ja2}}]
    Suppose $M$ is an essential $n$-manifold with $\pi=\pi_1(M)$ torsion-free, and let $\kappa\colon M\to K=K(\pi,1)$ be a classifying map that does not deform into the $(n-1)$-skeleton of $K(\pi,1)$. Then $M$ is said to satisfy the \emph{cap property} if $\kappa_*([M])\frown \mathfrak{a}_\pi\ne 0$. 
\end{definition}
\begin{remark}
 Since the orientation sheaf $\O_M$ on $M$ is determined by the kernel of the first Stiefel--Whitney class $w\colon \pi_1(M)=\pi\to \Z_2$, for $M$ to take (co)homology with coefficients in $\O_M$ is the same as taking (co)homology with coefficients in the twisted integers $\Z_w$.  Because $\kappa$ is a classifying map, we have $\kappa_*([M])\in H_n(K;\Z_w)$ so that $\kappa_*([M])\frown\mathfrak{a}_\pi \in H_0(K;\Z_w\otimes\F)$ for the $\pi$-module $\F$ as above.
\end{remark}

Our main result of this section studies $\dcat_1$ of connected sums of manifolds that have an essential summand satisfying the cap property. 
%For that, we prepare as follows.
%Suppose $M$ and $N$ are closed $n$-manifolds with a common closed $n$-disc $B$, whose interior is deleted from $M$ and $N$ to identify the resulting boundary spheres and form the connected sum $M\# N$. Let $\O_M$, $\O_N$, and $\O_B$ be the orientation sheaves on $M$, $N$, and $B$, respectively. Identifying $\O_M$ and $\O_N$ along $\O_B$ gives the orientation sheaf $\O'$ on $M \cup_B N$. Suppose $\psi\colon M \vee N \to M\cup_BN$ is a homotopy inverse of the quotient map $q\colon M\cup_BN\to(M\cup_BN)/B = M \vee N$. Then the orientation sheaf $\O$ on $M\vee N$ is obtained by pulling back $\O'$ along $\psi$. If $\iota\colon M\# N\to M\cup_B N$ denotes the inclusion, then the restriction of $\O'$ along $\iota$ gives the orientation sheaf $\O_{\#}$ on $M\# N$. In particular, $q\circ \iota$ takes $\O$ to $\O_{\#}$.

\begin{theorem}\label{th: essential sum}
    Let $M$ be a closed essential $n$-manifold ($n\ge 3$) such that $\pi_1(M)$ is torsion-free. If $M$ satisfies the cap property, then $\dcat_1(M \# N) = \cat_1(M\# N) = n$ for each closed $n$-manifold $N$. In particular, $\dcat_1(M)=n$.
\end{theorem}

To aid in the proof of this theorem, we first give the following technical remark.

\begin{remark}\label{rem: sheaf details}
Let $B$ be a small closed $n$-disc common to $M$ and $N$ that is used to form the connected sum $M\# N$. Consider the commutative diagram
\[
\begin{tikzcd}
    B \arrow[hookrightarrow]{r}{i_M} \arrow[hookrightarrow]{d}[swap]{i_N} 
    &
    M \arrow[hookrightarrow]{d}[swap]{j_M} 
    &
    M\#N \arrow[hookrightarrow]{dl}[swap]{\iota} \arrow{d}{q\circ\iota}
    \\
    N \arrow[hookrightarrow]{r}[swap]{j_N} 
    & 
    M\cup_BN \arrow{r}[swap]{q} 
    & 
    M\vee N,
\end{tikzcd}
\]
where $q$ is the quotient map $M\cup_BN\to(M\cup_BN)/B = M \vee N$, which is a homotopy equivalence since $B$ is contractible. Let $\O_X$ denote the orientation sheaf on a closed $n$-manifold $X$. The orientation sheaf $\O'$ on $M\cup_B N$ is obtained by identifying $\O_M$ and $\O_N$ along $\O_B$. Then the orientation sheaf $\O$ on $M\vee N$ is obtained by pulling back $\O'$ along a homotopy inverse $q'\colon M\vee N\to M \cup_B N$ of $q$, and the restriction of $\O'$ along $\iota$ gives the orientation sheaf $\O_{\#}$ on $M\# N$. In particular, all the maps in the above diagram preserve the respective orientation sheaves.
%$(q\circ\iota)^*(\O)=\O_{\#}$. Note also that the inclusions $i_M$, $i_N$, $j_M$, and $j_N$ preserve the respective orientation sheaves.
\end{remark}

\begin{proof}[Proof of Theorem~\ref{th: essential sum}]
    Let $\pi=\pi_1(M)$ and $\kappa\colon M\to K=K(\pi,1)$ be a classifying map that does not deform to the $(n-1)$-skeleton. Let $f\colon M\# N\to M$ be the degree $1$ map that collapses $N$ to a point. 
%It was shown in~\cite[Section 4]{Ja2} that $\dcat(\kappa\circ f)\ge n$. In detail, the 
The primary obstruction $\mathfrak{a}_\pi \in H^n(K,\F)$ to a section of $e_n^K$ gives the primary obstruction $(\kappa\circ f)^*(\mathfrak{a}_\pi)\in H^n(M\# N;(\kappa\circ f)^*(\F))$ to a lift of $\kappa\circ f$ along $e^K_n$. It was shown in~\cite[Section 4]{Ja2} that $(\kappa\circ f)^*(\mathfrak{a}_\pi)\ne 0$. In detail, we have a commutative diagram
\[
\begin{tikzcd}[contains/.style = {draw=none,"\in" description,sloped}]
&
M\# N \arrow[hookrightarrow]{r}{\iota} \arrow[swap]{d}{f} 
& 
M\cup_B N \arrow{d}{q}
\\
K
&
M   \arrow{l}[swap]{\kappa}
& 
M \vee N, \arrow{l}[swap]{d}
\end{tikzcd}
\]
where $d\colon M\vee N\to M$ is the collapsing map. Clearly, $f_*([M\#N])=[M]$. 
\\
\\

We now claim that $\iota_*([M\#N]) = [M]\oplus [N]$. By Remark~\ref{rem: sheaf details}, we have that $\iota_*([M\#N])\in H_n(M\cup_BN;\O')$. As in~\cite[Chapter~VI]{Wh}, homology with local coefficients satisfies the axioms of excision and the long exact sequence of a pair, so by the argument given in~\cite[Chapter~2,~Section~3]{Ha}, we get a Mayer--Vietoris sequence\hspace{0.3mm}\footnote{\hspace{0.3mm}This was pointed out by Mark Grant in the MathOverflow post number \href{https://mathoverflow.net/questions/124278/mayer-vietoris-sequence-in-homology-with-local-coefficients}{124278}.}
\begin{multline*}
    \cdots\to H_n(B;\O_B)\to H_n(M;\O_M)\oplus H_n(N;\O_N) \to H_n(M\cup_BN;\O')
    \\
    \to H_{n-1}(B;\O_B)\to\cdots,
\end{multline*}
where $H_n(B;\O_B)=H_{n-1}(B;\O_B)=0$ as $n>2$. Thus, we have an isomorphism $H_n(M\cup_BN;\O')\cong H_n(M;\O_M)\oplus H_n(N;\O_N)$ and so $\iota_*([M\#N]) = [M]\oplus [N]$. 

Note that $\iota$ induces an epimorphism of fundamental groups by definition. Since $q$ is a homotopy equivalence, the induced map 
\[
    (q\circ \iota)_*:H_0(M\# N;(q\circ \iota)^*(\A)) \to H_0(M\vee N;\A)
    \]
is an isomorphism for each system of local coefficients $\A$. Here, we recall from Remark~\ref{rem: sheaf details} that $(q\circ\iota)^*(\O)=\O_{\#}$. Therefore, in $\O\otimes (\kappa\circ d)^*(\F)$ coefficients, we have that
\begin{multline*}
(q \circ \iota)_*\left([M\#N]\frown (\kappa\circ f)^*(\mathfrak{a}_\pi)\right)
= (q \circ \iota)_*\left([M\#N]\frown (q \circ \iota)^* \circ \left(\kappa\circ d)^*(\mathfrak{a}_\pi\right)\right)
\\
= (q \circ \iota)_*([M\# N])\frown \left(\kappa\circ d)^*(\mathfrak{a}_\pi\right)
= ([M]\oplus [N])\frown(\kappa^*(\mathfrak{a}_\pi)\oplus \alpha),
\end{multline*}
where $\alpha\in H^n(N;\Z)$ is a class such that $(\kappa\circ d)^*(\mathfrak{a}_\pi)=d^*(\kappa^*(\mathfrak{a}_\pi))=\kappa^*(\mathfrak{a}_\pi)\oplus \alpha$. We can write $d^*(\kappa^*(\mathfrak{a}_\pi))\in H^n(M\vee N;(\kappa\circ d)^*(\F))$ as a direct sum because there is a Mayer--Vietoris sequence for cohomology with local coefficients (see~\cite[Chapter~II,~Section~13]{Bre}), so using similar arguments as before, we get an isomorphism 
\[
H^n(M\vee N;(\kappa\circ d)^*(\F))\cong H^n(M;\kappa^*(\F))\oplus H^n(N;\Z)
\]
since the restrictions of the coefficient system $(\kappa\circ d)^*(\F)$ to $M$ and $N$ are $\kappa^*(\F)$ and trivial, respectively.
%, and hence we can write $d^*(\kappa^*(\mathfrak{a}_\pi))=\kappa^*(\mathfrak{a}_\pi)\oplus\alpha$.
%\textcolor{blue}{Now, how to justify that $H^n(M\vee N;\O)\cong H^n(M;\O_M)\oplus H^n(N;\O_N)$, so that we can formally get $d^*(\kappa^*(\mathfrak{a}_\pi))=\kappa^*(\mathfrak{a}_\pi)\oplus \alpha$? The map above $q$ gives us $H^n(M\vee N;\O)\cong H^n(M\cup_BN;\O')$. But, the Mayer--Vietoris sequence above is for homology with local coefficients, not for cohomology, right?  Poincaré duality (see Ch. V, Sect.~9-10 of Bredon) says $H^n(M\cup_BN;\O')\cong H_0(M\cup_BN;\O')$. But then we don't know whether $H_0(M\cup_BN;\O')\cong H_0(M;\O_M)\oplus H_0(N;\O_N)$, which would have given $\cong H^n(M;\O_M)\oplus H^n(N;\O_N)$....?}
The cap property on $M$ ensures that $\kappa_*([M]\frown \kappa^*(\mathfrak{a}_\pi))\ne 0$ in $\O_M\otimes \kappa^*(\F)$ coefficients, so that $([M]\oplus [N])\frown(\kappa^*(\mathfrak{a}_\pi)\oplus \alpha) \ne 0$, which implies $(\kappa\circ f)^*(\mathfrak{a}_\pi)\ne 0$ in view of the above equality. Hence, there is no lift of $\kappa\circ f$ along $e^K_n$ since $(\kappa\circ f)^*(\mathfrak{a}_\pi)$ is the only obstruction to such a lift. But then $\dcat(\kappa\circ f)\ge n$ by definition (\emph{cf.} Section~\ref{sec:basic2}). Now, Propositions~\ref{prop:min1cat1},~\ref{prop:min}, and~\ref{prop:XKpi1} imply
\[
n\le\dcat(\kappa\circ f)=\dcat_1(\kappa\circ f)\le\dcat_1(M\# N)\le\cat_1(M\# N)\le n,
\]
where the last inequality is obvious. This completes the proof.
\end{proof}

A special case of our interest is as follows.

\begin{corollary}\label{cor: essential sum1}
Let $M$ be a closed essential $n$-manifold such that $K(\pi_1(M),1)$ is a closed $n$-manifold. Then $\dcat_1(M \# N) = \cat_1(M\# N) = n$ for each closed $n$-manifold $N$. In particular, this holds if $M$ is aspherical.
\end{corollary}
\begin{proof}
    From obstruction theory, it follows that $\kappa_*([M])\ne 0$,~\cite[Theorem 8.2]{Ba}. Therefore, $M$ satisfies the cap property due to Poincaré duality in $K(\pi_1(M),1)$, see~\cite[Corollary 3.4]{Ja2}. Hence, Theorem~\ref{th: essential sum} applies.
\end{proof}

\begin{remark}\label{Alex}
    One can generalize Corollary~\ref{cor: essential sum1} beyond the realm of manifolds by taking $N$ to be a closed (non-manifold) Alexandrov $n$-space~\cite{BGP} whose curvature is bounded from below. Examples of such spaces include supsensions of $\R P^2$ and $\C P^2$ and their products (and more generally, orbit spaces of isometric compact Lie group actions on complete Riemannian manifolds whose sectional curvatures are uniformly bounded below). In case of such $N$, the connected sum $M\# N$ along any \emph{regular point} of $N$ (see~\cite{BGP,decomposition}) still makes sense, and one can obtain the equality $\dcat_1(M\#N)=\cat_1(M\#N)=n$ by proceeding as in Theorem~\ref{th: essential sum} and using ideas from~\cite[Section 7]{Ja2}.
\end{remark}

The results above provide examples of compact metric spaces whose $\dcat_1$ values can be as large as desired.

\begin{remark}\label{rmk:sum with rpn}
    We emphasize that we only require assumptions on $\pi_1(M)$ in the above statements; there are no assumptions on $\pi_1(M\#N)$. For example, we have $\dcat_1(M\# \R P^n)=n$ for any aspherical manifold $M$, even though $\dcat_1(\R P^n)=1$ and $K(\pi_1(M\#\R P^n))\simeq M\vee \R P^{\infty}$ for each $n$.
\end{remark}

We have from Corollary~\ref{cor:classmap} and Theorem~\ref{th: after kw} that $\dcat_1(M)=\dim(M)$ if $M$ is a closed aspherical manifold. So, Theorem~\ref{th: essential sum} and Corollary~\ref{cor: essential sum1} are extensions of this result to some essential manifolds and connected sums having aspherical summands, respectively.

\begin{remark}
    Since $\dcat_1(X)\le\dcat(X)\le\dim(X)$ is always true, our results recover the corresponding results of~\cite{Ja2} on distributional category. Concretely, Theorems~A,~B,~C,~and F in~\cite{Ja2} and their corollaries in Sections~3 and~4 there are extended in this section without using those results.
\end{remark}

\subsection{Some computations}\label{subsec:appl}
Let us consider a particular class of smooth, essential manifolds. We say that a closed symplectic manifold $M$ is \emph{symplectically aspherical} with respect to a symplectic form $\omega$ if for any smooth map $f\colon S^2\to M$, one has
\[
\int_{S^2}f^*\omega=0.
\]
This is equivalent to the algebraic condition that $[\omega]\circ h\colon \pi_2(M)\to\R$ vanishes, where $h\colon \pi_2(M)\to H_2(M;\Z)$ is the Hurewicz homomorphism and $[\omega]\in H_{\text{dR}}^2(M;\R)$ corresponds to $\omega$. It is well-known that symplectically aspherical manifolds are essential (in fact, rationally essential), see~\cite{RO}. So, we have the following.
\begin{corollary}\label{cor: symp asph}
    If $M$ is a symplectically aspherical $n$-manifold and $K(\pi_1(M),1)$ is a closed $n$-manifold, then $\dcat_1(M\# N)=\cat_1(M\# N)=n$ for each $N$.
\end{corollary}

\begin{ex}
There are non-trivial symplectically aspherical manifolds with $\pi_2\ne 0$ due to~\cite{Go}. Moreover, there are manifolds with $\pi_2\ne 0$ that satisfy the hypotheses of Corollary~\ref{cor: symp asph} --- we refer the reader to~\cite[Section 3.2]{Ja2} for various such examples in all even dimensions $\ge 4$. We note that for all these closed manifolds, our computation of their distributional one-category is new, in the sense that it does not follow from Corollary~\ref{cor:classmap}, or any other known results on the distributional invariants.
\end{ex}

We conclude this section by looking at another application of Corollary~\ref{cor: essential sum1} to the distributional one-category of closed $3$-manifolds.

Recall from~\cite{GLGA} (see also~\cite{OR}) that if $M$ is a closed $3$-manifold, then
\[
\cat(M)=\begin{cases}
    1 & \text{ if } \pi_1(M)=0 \\
    2 & \text{ if } \pi_1(M)\ne 0 \text{ is free} \\
    3 & \text{ if } \pi_1(M) \text{ is not free}
\end{cases}\hspace{4mm} \text{and} \hspace{4mm}
\cat_1(M)=\begin{cases}
    0 & \text{ if } \pi_1(M)=0 \\
    1 & \text{ if } \pi_1(M)\ne 0 \text{ is free} \\
    3 & \text{ if } \pi_1(M) \text{ is not free}
\end{cases}
\]
Moreover, if $\pi_1(M)$ is torsion-free, then $\dcat(M)=\cat(M)$, see~\cite{Ja2}.

\begin{proposition}\label{prop:3-manifolds}
    If $M$ is a closed $3$-manifold, then
    \[
\dcat_1(M)=\cat_1(M)=\begin{cases}
    0 & \text{ if } \pi_1(M)=0 \\
    1 & \text{ if } \pi_1(M)\ne 0 \text{ is free} \\
    3 & \text{ if } \pi_1(M) \text{ is not free but torsion-free.}
\end{cases}
\]
\end{proposition}
\begin{proof}
    The equality in the first case is trivial. When $\pi_1(M)\ne 0$ is free, then $\dcat_1(M)=\cat_1(M)=1$ is obtained from Example~\ref{ex:smallcd}. Let us now assume that $\pi_1(M)$ is not free but torsion-free. In this case, from the prime decomposition of closed $3$-manifolds and Papakyriakopoulos's sphere theorem~(see~\cite{He}), it turns out that $M$ has a closed aspherical summand (see also~\cite[Section 5.1]{Ja2} for a LS-categorical proof of this fact). Therefore, Corollary~\ref{cor: essential sum1} gives $\dcat_1(M)=3$.
\end{proof}

The converse of Proposition~\ref{prop:3-manifolds} is not true (\emph{cf.} Remark~\ref{rmk:sum with rpn}).

\begin{corollary}\label{cor:nocd=2}
 If $M$ is a closed $3$-manifold such that $\pi_1(M)$ is not free, then $\cd(\pi_1(M))\ge 3$.
 %Equality holds if $\pi_1(M)$ is torsion-free.
\end{corollary}
\begin{proof}
First, note that $\cd(\pi_1(M))\ge 2$ by the Stallings--Swan theorem. 
%We now explain why $\cd(\pi_1(M))\ne 2$.
Now, let us suppose $\cd(\pi_1(M))=2$. Then we must have $\dcat_1(M)=\cat_1(M)=2$ by Corollary~\ref{cor:cd1}. But $\cd(\pi_1(M))=2$ means $\pi_1(M)$ is torsion-free, so Proposition~\ref{prop:3-manifolds} gives $\dcat_1(M)=\cat_1(M)=3$, a contradiction. Thus, $\cd(\pi_1(M))\ne 2$. 
\end{proof}

In fact, it follows from the prime decomposition~\cite[Theorems 3.15 and 3.21]{He} that if $\pi_1(M)$ is not free, then $\cd(\pi_1(M))=3$, unless $\pi_1(M)$ has torsion, in which case $\cd(\pi_1(M))$ is infinite. However, we obtained this in Corollary~\ref{cor:nocd=2} directly using our invariant $\dcat_1$, which is interesting.

\begin{remark}\label{rem:torsion}
    In the case when $\pi_1(M)$ has non-trivial torsion, finding $\dcat_1(M)$ may be difficult. In particular, there is no single answer as $\dcat_1(\R P^3)=1$ while $\dcat_1(M\# \R P^3)=3$ for any aspherical $3$-manifold $M$. If $\pi_1(M)=\Z_3$, then $\dcat_1(M)\in\{1,2\}$ by Corollary~\ref{cor:kunivcover}, but it is unclear if anything more can be said. We do note, however, that if $\dcat_1(M)=2$, then $\cat_1(M)\ge 2$ and hence, $\cat_1(M)=3$ by~\cite{GLGA}, so we can get examples of closed $3$-manifolds $M\not\simeq \R P^3$ with $\dcat_1(M)\ne\cat_1(M)$. But we don't know yet whether $\dcat_1(M)=2$ is a possibility for $3$-manifolds. Some manifolds of interest can be $\R P^2\times S^1$ and $\R P^3\# \R P^3$, whose $\dcat$ values are also unknown,~\cite{Ja2}.
\end{remark}

\section{Sequential diagonal distributional complexity}\label{sec: diag distrib TC}
In~\cite{FGLO1, FGLO2}, a new invariant was introduced to provide a lower bound on topological complexity (\emph{cf.} Section~\ref{subsec:secat section}). This invariant is the diagonal topological complexity $\TC^{\D}$. Soon after, the sequential versions of $\TC^{\D}$, denoted $\TC^{\D}_m$ for $m\ge 2$, were introduced in~\cite{FO} and their theory was further developed in~\cite{PS}. We discuss these invariants here because they are the $\TC$-analogues of $\cat_1$. 
Indeed, for $m\ge 2$, we always have
\begin{equation}\label{five}
 \cat_1(X^{m-1})\le\TC^{\D}_m(X)\le\cat_1(X^m),   
\end{equation}
as well as $\TC_m^{\D}(X) \leq \TC_m(X)$. The definition of $\TC_m^{\D}(X)$ can be given as follows. Any path $\gamma\in X^I$ may
be lifted uniquely to the universal cover $\widetilde X$ once a starting point is chosen, and then the ending point
is determined by the uniqueness of lifting. In order to do away with choosing starting points, we can take the 
(free) diagonal action of $\pi=\pi_1(X)$ on the product $\widetilde X \times \cdots \times \widetilde X=\Wi{X}^m$, where the product is taken $m$-times, and form the orbit space
$\widetilde X \times_\pi \cdots\times_\pi \widetilde X$, which we will denote by $\Wi{X}^m_\pi$. This is a covering of $X^m$ corresponding to the diagonal subgroup $\Delta\subset \pi^m$ with fiber $\pi^{m-1}$, and there is a commutative diagram 
\begin{equation}\label{eq: diag 1}
    \xymatrix{
P(X) \ar[rr]^-\phi \ar[dr]_-{\text{ev}_m^X} & & \Wi{X}^m_\pi \ar[dl]^-{r_m^X} \\
& X^m, &
}
\end{equation}
where $P(X)$ is the free path space of $X$, $\text{ev}^X_m$ is the evaluation fibration defined in Section~\ref{subsec:secat section}, $r^X_m$ is the covering map, and
\[
\phi(\gamma) = \left[\gamma(0),\gamma\left(\frac{1}{m-1}\right),\ldots,\gamma\left(\frac{m-2}{m-1}\right),\gamma(1)\right]
\]
is well-defined because we quotient by the $\pi$-action. Then we define the $m$\emph{-th sequential diagonal topological complexity} as
\[
\TC_m^{\D}(X) = \secat(r_m^X).
\]
In this paper, we define and study the following.

\begin{definition}\label{def:dTCD}
The $m$\emph{-th sequential diagonal distributional complexity} of a space $X$, denoted $\dTC_m^{\D}(X)$, is the least integer $k$ such that the
fibration 
\[
(\Wi{X}^m_\pi)_{k+1} \xrightarrow{(r_m^X)_{k+1}} X^m
\]
admits a section. Here, the fibration $(r_m^X)_{k+1}$ has fiber $\B_{k+1}(\pi^{m-1})$.
\end{definition}
In other words, $\dTC_m^{\D}(X)=\dsecat(r_m^X)$, so that $\dTC_m^{\D}(X)\le\TC_m^{\D}(X)$ for each $m\ge 2$ follows immediately, see Section~\ref{sec:basic2}. 

\subsection{Standard inequalities}
We begin by establishing some properties of $\dTC_m^{\D}(X)$ in analogy with certain 
results for $\TC_m^{\D}(X)$ discussed in~\cite{FO,PS}.

\begin{proposition}\label{prop: dtcD vs dtc}
    For each $m\ge 2$, we have that $\dTC_m^{\D}(X)\le\dTC_m(X)$. The inequality is saturated if $X=K(\pi,1)$.
\end{proposition}

\begin{proof}
    Suppose $\dTC_m(X)=k$. Then the diagram in~\eqref{eq: diag 1} induces a commutative diagram
    $$
    \xymatrix{
P(X)_{k+1} \ar[rr]^-{\phi_{k+1}} \ar[dr]_-{(\text{ev}_m^X)_{k+1}} & & (\Wi{X}^m_\pi)_{k+1} \ar[dl]^-{(r_m^X)_{k+1}} \\
& X^m. &
}
    $$
    So, a section $s\colon X^m\to P(X)_{k+1}$ of $(\text{ev}_m^X)_{k+1}$ gives a section $\phi_{k+1}\circ s$ of $(r_m^X)_{k+1}$. This implies $\dTC_m^{\D}(X)\le k$. If $X=K(\pi,1)$, then $\phi$ is a fiber homotopy equivalence. Therefore, $\phi_{k+1}$ is a fiber homotopy equivalence as well that fits in the above diagram and so, any section of $(r_m^X)_{k+1}$ determines a section of $(\text{ev}_m^X)_{k+1}$, hence giving the inequality $\dTC_m(K(\pi,1))\le\dTC_m^{\D}(K(\pi,1))$.
\end{proof}

\begin{proposition}\label{prop: dtcD vs dcat1}
    For each $m\ge 2$, we have $\dcat_1(X^{m-1})\le\dTC^{\D}_m(X)\le\dcat_1(X^m)$.
\end{proposition}

\begin{proof}
    Consider the following commutative diagram:
    \[
    \begin{tikzcd}
        Q \arrow{r} \arrow{d}[swap]{\rho}
        &
        \Wi{X}^m_\pi \arrow{d}[swap]{r_m^X}
        &
        \Wi{X}^m \arrow{l} \arrow{dl}{q^{X^m}}
        \\
        X^{m-1} \arrow{r}{\iota} 
        & 
        X^m.
        &
    \end{tikzcd}
    \]
    Here, $\iota$ is the inclusion of the first $(m-1)$-coordinates, the square on the left is a pullback, and the horizontal arrow on the right is the natural quotient map. First, by the triangle on the right, we get using the proof technique of Proposition~\ref{prop: dtcD vs dtc} that
    \[
    \dTC_m^{\D}(X)=\dsecat(r_m^X)\le \dsecat(q^{X^m})=\dcat_1(X^m).
    \]
    Now, we look at the pullback space $Q$. At the level of fundamental groups, we have that $\Im(\iota_{\#})=\pi^{m-1}\times 1\subset \pi^m$ and $\Im((r_m^X)_{\#})=\Delta\subset \pi^m$. These subgroups generate all of $\pi^m$, so by the exact sequence of homotopy groups, we see that $Q$ is a connected covering space of $X^{m-1}$. Furthermore, the fiber of $Q \to X$ is $\pi^{m-1}$ by the pullback property. But this means $Q \simeq \Wi{X}^{m-1}$ and $\rho\colon Q \to X^{m-1}$ is compatible with the universal covering $q^{X^{m-1}}\colon \Wi{X}^{m-1} \to X^{m-1}$. In other words, $q^{X^{m-1}}$ is, up to homotopy, the pullback of $r_m^X$ along $\iota$. Therefore, applying Lemma~\ref{lem:dsecat1}, we get
    \[
    \dcat_1(X^{m-1})=\dsecat(q^{X^{m-1}}) \le \dsecat(r_m^X)=\dTC_m^{\D}(X).
    \]
    This completes the proof.
\end{proof}

\begin{remark}
Proposition~\ref{prop: dtcD vs dcat1} shows that the sequence $\{\dTC^{\D}_m(X)\}_{m\ge 2}$ is non-decreasing for any $X$, so the values $\dTC^{\D}_m(X)$ can be arbitrarily large for each $m\ge 2$ (\emph{cf.} Remark~\ref{new addition} and Section~\ref{subsec: ess + cap}).
\end{remark}

We end this subsection by showing that the lower bound in Proposition~\ref{prop: dtcD vs dcat1} is attained in the case of CW $H$-spaces (compare with~\cite[Proposition 3.13]{PS} for $\TC^{\D}_m$, and with~\cite[Theorem 6.3]{Ja1} for $\dTC_m$).

\begin{proposition}\label{prop:Hspace}
If $X$ is a CW $H$-space, then for each $m\ge 2$, we have that
\[
\dTC_m^{\D}(X) = \dcat_1(X^{m-1}).
\]
\end{proposition}
\begin{proof}
The loop structure of $[X,X]$, due to James, shows that there is a right inverse map
$\eta\colon X \to X$ with the property that
\[
X\xrightarrow{\Delta} X\times X\xrightarrow{1_X\times \eta} X\times X\xrightarrow{m} X
\]
is null-homotopic, where $m$ is the multiplication on $X$ corresponding to a $H$-space structure. Define $h=m\circ (1_X \times \eta)\colon X\times X\to X$. On homology, this map induces $h_*(a,b)=a-b$ and since, for an $H$-space, $\pi_1(X)=H_1(X)$, we see that the map of fundamental groups is $h_\#(a,b)=a-b$ as well. Define $f_m\colon X^{m}\to X^{m-1}$ such that
\[
f_m\left(x_1,\ldots,x_{m-1},x_m\right)=\left(h\left(x_1,x_m\right),\ldots,h\left(x_{m-1},x_m\right)\right).
\]
It follows that the induced map $(f_m)_\#\colon \pi^m\to \pi^{m-1}$ is such that
\[
(f_m)_\#\left(a_1,\ldots,a_{m-1},a_m\right)=\left(h_\#\left(a_1,a_m\right),\ldots,h_\#\left(a_{m-1},a_m\right)\right).
\]
Hence, $\Ker((f_m)_\#)=\Delta\subset \pi^m$. Therefore, we can lift $f_m\colon X^m\to X^{m-1}$ to a map $\Wi{f_m}\colon \Wi{X}^m_\pi\to \Wi{X}^{m-1}$ in the diagram
\[
\begin{tikzcd}
    \Wi{X}^m_\pi \arrow{r}{\Wi{f_m}} \arrow{d}[swap]{r_m^X}
    &
    \Wi{X}^{m-1}\arrow{d}{q^{X^{m-1}}}
    \\
    X^m\arrow{r}{f_m}
    &
    X^{m-1}.
\end{tikzcd}
\]
Note that $\Im((f_m)_\#)=\pi^m/\Delta\cong\pi^{m-1}$,  and this in turn shows that $\Wi{f_m}$ restricted to each fiber $\pi^m/\Delta\cong\pi^{m-1}$ is a bijection. This then implies that the diagram above
is a pullback. But then we get from Lemma~\ref{lem:dsecat1} that
\[
    \dTC_m^{\D}(X)=\dsecat(r_m^X)\le \dsecat(q^{X^{m-1}})=\cat_1(X^{m-1}).
\]
The reverse inequality follows from Proposition~\ref{prop: dtcD vs dcat1}.
\end{proof}

\subsection{Behavior on $\pi_1$-isomorphisms and retracts}
Motivated by Corollary~\ref{cor:isoineq}, we study the behavior of $\dTC^{\D}_m$ on maps that induce isomorphisms of fundamental groups. For that, we prepare as follows.
\begin{lemma}\label{lem:diagpull}
For a map $f\colon X \to Y$ which induces an isomorphism of fundamental groups, the following diagram is a pullback:
\[
\begin{tikzcd}
    \Wi{X}^m_\pi \arrow{r}{\Wi{f}^m} \arrow{d}[swap]{r^X_m}
    &
    \Wi{Y}^m_\pi \arrow{d}{r^Y_m}
    \\
    X^m\arrow{r}{f^m}
    & 
    Y^m.
\end{tikzcd}
\]
Here, $f^m=f\times\cdots\times f$ (multiplied $m$-times) and $\Wi{f}^m$ denotes the map induced on the $\pi$-quotients by the map on the universal covering spaces.
\end{lemma}
\begin{proof}
Let $Q=\{(x_1,\ldots,x_m,[\alpha_1,\ldots,\alpha_m])\mid\, f(x_i)=q^Y(\alpha_i) \text{ for all }i\le m\}$ be the pullback of $r_m^Y$ along $f^m$. Define a map $\Theta\colon \Wi{X}^m_\pi \to Q$ by
\[
\Theta([\sigma_1,\ldots,\sigma_m]) =  (q^X(\sigma_1),\ldots, q^X(\sigma_m),[\Wi f(\sigma_1),\ldots,\Wi f(\sigma_m)]),
\]
where $\Wi{f}\colon \Wi{X}\to \Wi{Y}$ is the natural map on universal covers. Clearly, $\Theta$ is well-defined since $f \circ q^X = q^Y \circ \Wi f$. Since $f$ induces an isomorphism of fundamental groups, we see that
$\Wi f$ induces bijections of 
fibers $(q^X)^{-1}(x)$ and $(q^Y)^{-1}(f(x))$ for all $x\in X$. Moreover, 
the fibers of $r_m^X$ and $r_m^Y$ are both given by
$\pi^m/\Delta \cong \pi^{m-1}$, where $\Delta$ is the diagonal subgroup of $\pi^m$. Therefore, we can define a map $\phi\colon Q \to \Wi{X}^m_\pi$ by
\[
\phi(x_1,\ldots,x_m,[\alpha_1,\ldots,\alpha_m]) = [\sigma_1,\ldots,\sigma_m],
\]
where $q^X(\sigma_i)=x_i$ and $\Wi f(\sigma_i) = \alpha_i$ for each $i\le m$. If we take the representative $[g\cdot \alpha_i,\ldots, g\cdot \alpha_m]$, we get $[f_\#^{-1}(g)\cdot \sigma_1,\ldots,f_\#^{-1}(g)\cdot \sigma_m]$ since $\Wi f$ is $f_\#$-equivariant (here, we again use the isomorphism of fundamental groups). 
But this element is the same as $[\sigma_1,\ldots,\sigma_m]$ because we quotient by the diagonal action. Hence, $\phi$ is well-defined. We now have
\begin{align*}
\Theta\circ \phi(x_1,\ldots,x_m,[\alpha_1,\ldots, \alpha_m]) 
& 
= \Theta([\sigma_1,\ldots,\sigma_m]) 
\\
& 
= (q^X(\sigma_1),\ldots,q^X(\sigma_m),[\Wi f(\sigma_1),\ldots,\Wi f(\sigma_m)]) 
\\
& 
= (x_1,\ldots,x_m,[\alpha_1,\ldots,\alpha_m])
\end{align*}
by the definition of $\phi$. We also have
\begin{align*}
\phi\circ \Theta\left([\sigma_1,\ldots,\sigma_m\right]) 
& 
= \phi(q^X(\sigma_1),\ldots,q^X(\sigma_m),[\Wi f(\sigma_1),\ldots,\Wi f(\sigma_m)]) 
\\
& 
= \left[\sigma_1,\ldots,\sigma_m\right]
\end{align*}
by the definition of $\phi$ again. Hence, we get that $Q \cong \Wi{X}^m_\pi$.
\end{proof}

\begin{remark}\label{rem:homopull}
We have shown that the diagram of Lemma~\ref{lem:diagpull} is a pullback 
directly, but we could also proceed as follows. The fact that $f$ induces an isomorphism
on fibers of $r_m^X$ and $r_m^Y$ implies that the diagram is a homotopy pullback. 
But since $r_m^Y$ is a fibration, the pullback itself is a homotopy pullback. 
Therefore, up to homotopy, we could replace $Q$ by $\Wi{X}^m_\pi$. This would lead us to look at homotopy sections, but in the case of fibrations,
we could then obtain true sections (\emph{cf.} Remark~\ref{rmk:homotopy section}). 
\end{remark}

\begin{corollary}\label{cor:dTCDpull}
For a map $f\colon X \to Y$ which induces an isomorphism of fundamental groups, the following diagram is a homotopy pullback 
for each $n$:
\[
\begin{tikzcd}[column sep=large, row sep=large]
(\Wi{X}^m_\pi)_n \arrow{r}{(\Wi{f}^m)_n} \arrow{d}[swap]{(r^X_m)_n}
    &
    (\Wi{Y}^m_\pi)_n \arrow{d}{(r^Y_m)_n}
    \\
    X^m\arrow{r}{f^m}
    & 
    Y^m.
\end{tikzcd}
\]
\end{corollary}
\begin{proof} 
From Lemma~\ref{lem:diagpull}, the lift $\Wi f^m$ induces a bijection of the discrete fibers of $r_m^X$ and $r_m^Y$, so $(\Wi f^m)_n$ induces a bijection of the fibers $\B_n(\pi^{m-1})$ of the fibrations $(r_m^X)_n$ and $(r_m^Y)_n$ for each $n$. 
\end{proof}

\begin{proposition}\label{prop:pi1isodTCD}
For a map $f\colon X \to Y$ which induces an isomorphism of fundamental groups,
we have for each $m\ge 2$ that
$$\dTC_m^{\D}(X) \leq \dTC_m^{\D}(Y).$$
\end{proposition}
\begin{proof}
Suppose $\dTC^{\D}_m(Y)=n$. The diagram of Corollary~\ref{cor:dTCDpull} is a homotopy pullback for $n+1$, so a section of $(r^Y_m)_{n+1}$ gives a homotopy section of $(r^X_m)_{n+1}$, which then gives us a true section of $(r^X_m)_{n+1}$, thereby implying $\dTC^{\D}_m(X)\le n$.
\end{proof}

\begin{proposition}\label{prop:hretractdTCD}
If $A$ is a homotopy retract of $X$, then for each $m\ge 2$,
$$\dTC_m^{\D}(A) \leq \dTC_m^{\D}(X).$$
\end{proposition}

\begin{proof}
    There are maps $i\colon A \hookrightarrow X$ and $\rho\colon X \to A$ with $\rho\circ i \simeq 1_A$. Note that on fundamental groups, $i_\#$ is injective with image $H = \Im(i_\#) \subset \pi=\pi_1(X)$ and $\rho_\#$ is surjective with
$\rho_{\#|_H} = 1_H$. The induced map on universal covers $\Wi \rho
\colon \widetilde X \to \widetilde A$ obeys $\Wi \rho(g\cdot \alpha) = 
\rho_\#(g)\cdot \Wi\rho(\alpha)$ for all $g \in \pi$. This gives a map $\bar\rho\colon \Wi{X}^m_\pi\to\Wi{A}^m_H$. Now, suppose $\dTC^{\D}_m(X)=n$. Consider the following 
commutative diagram:
\[
\xymatrix{
& (\Wi{X}^m_\pi)_{n+1} \ar[d]^-{(r^X_m)_{n+1}} \ar[r]^-{\bar\rho_{n+1}} & (\Wi{A}^m_H)_{n+1}
\ar[d]^-{(r^A_m)_{n+1}} \\
A^m \ar[ur]^-{s\circ i^m} \ar[r]^-{i^m} & X^m \ar@/^0.2cm/[u]^-s \ar[r]^-{\rho^m} & A^m,
}
\]
where $i^m$ and $\rho^m$ are the natural product maps, and $s\colon X^m\to (\Wi{X}^m_\pi)_{n+1}$ is a section to $(r^X_m)_{n+1}$. Then we have
\[
(r^A_m)_{n+1} \circ \bar\rho_{n+1} \circ s\circ i^m = \rho^m \circ (r^X_m)_{n+1}\circ s\circ i^m=\rho^m\circ i^m=(\rho \circ i)^m\simeq 1_{A^m}.
\]
Therefore, we have a section of the fibration $(r^A_m)_{n+1}$. This gives us the desired inequality $\dTC^{\D}_m(A)\le n$.
\end{proof}

An immediate consequence of Proposition~\ref{prop:hretractdTCD} is the following.

\begin{corollary}
For each $m\ge 2$, $\dTC^{\D}_m$ is a homotopy invariant.
\end{corollary}

\subsection{Relationship with the fundamental group}
Let us consider a classifying map $\kappa\colon X \to K(\pi,1)$, where $\pi_1(X)
= \pi$. By Proposition~\ref{prop:pi1isodTCD}, we have 
\begin{equation}\label{four}
    \dTC_m^{\D}(X) \leq \dTC_m^{\D}(\pi),
\end{equation}
where, as usual, we write $\dTC_m^{\D}(\pi)$ for $\dTC_m^{\D}(K(\pi,1))$. In this section, we explore how close (and far apart) the quantities $\dTC_m^{\D}(X)$ and $\dTC_m^{\D}(\pi)$ can be.

\begin{ex}\label{ex: symmetric products}
For each $m\ge 2$, $\dTC^{\D}_m(X)$ and $\dTC_m^{\D}(\pi)$ can be arbitrarily far apart. As in Remark~\ref{rmk:dcat vs dcat-pi}, here also symmetric products of closed orientable surfaces give natural examples. Let us fix $m\ge 2$ and consider $X=SP^n(M_g)$ for $g > \tfrac{mn}{m-1}$ so that $\pi=\Z^{2g}$. It is known due to~\cite{Ja3} that $\dTC_m(X)=2mn$ in this case. Here, because $T^{2g}\simeq K(\pi,1)$, we get by Proposition~\ref{prop: dtcD vs dtc} that
\[
\dTC^{\D}_m(X)\le\dTC_m(X)=2mn<2g(m-1)=\dTC_m(T^{2g})=\dTC^{\D}_m(\pi),
\]
where the second last equality comes from~\cite[Proposition 7.8]{Ja1}. 
\end{ex}

At the other end of the spectrum, we have the following simple criterion for $\dTC_m^{\D}(X)$ and $\dTC_m^{\D}(\pi)$ to coincide.
\begin{theorem}\label{thm:cdconstraintdTCD}
For $X$ with $\pi_1(X)=\pi$, suppose that $\Wi X$ is $(k-1)$-connected and that $\cd(\pi) \leq k$. Then for each $m\ge 2$, 
\[
\dTC_m^{\D}(X) = \dTC_m^{\D}(\pi) = \dTC_m(\pi).
\]
\end{theorem}

\begin{proof}
The same argument as in the proof of Theorem~\ref{theo:mainresult2} shows that  
$K=K(\pi,1)$ is a homotopy retract of $X$. Thus, $\dTC_m^{\D}(\pi)=\dTC_m^{\D}(K) \leq \dTC_m^{\D}(X)$ holds by Proposition~\ref{prop:hretractdTCD}. The reverse inequality follows from Proposition~\ref{prop:pi1isodTCD}, giving the desired equality.
\end{proof}

As in the case of $\dcat_1$ (see Section~\ref{sec: relationship with pi1}), we get the following two immediate consequences of Theorem~\ref{thm:cdconstraintdTCD}.

\begin{corollary}\label{cor:dTCD2}
    For $X$ with $\pi_1(X)=\pi$, suppose that $\widetilde X$ is $(k-1)$-connected for $k\ge 2$ and that $\cd(\pi) \leq k$. If $Z$ is any $(k-1)$-connected space, then for each $m\ge 2$,
    \[
    \dTC^{\D}_m(X\times Z)=\dTC_m(\pi)=\dTC^{\D}_m(X).
    \]
\end{corollary}

\begin{corollary}\label{cor:dTCD1}
For $X$ with $\pi_1(X)=\pi$, suppose $\cd(\pi)\le 2$. Then we have for each $m\ge 2$ that $\dTC^{\D}_m(X)=\dTC_m(\pi)$.
\end{corollary}

Under the hypotheses of Theorem~\ref{thm:cdconstraintdTCD}, $\cd(\pi^{m-1})\le \dTC^{\D}_m(X)\le\cd(\pi^m)$ in view of Theorem~\ref{theo:mainresult2} and Proposition~\ref{prop: dtcD vs dcat1}. The following example shows that these estimates on $\dTC^{\D}_m(X)$ are sharp!
\begin{ex}
If the fundamental group $\pi$ of $X$ is a surface group $\pi=\pi_1(M_g)$ for $g\ge 1$, then $\dTC^{\D}_m(X)=\dTC_m(M_g)$ by Corollary~\ref{cor:dTCD1}. For $g\ge 2$, we have $\dTC_m(M_g)=2m$ by~\cite[Proposition 7.5]{Ja1}, so that $\dTC_m^{\D}(X)=\cd(\pi^m)$ in this case. On the other hand, if $g=1$, we have $\dTC^{\D}_m(X)=2m-2=\cd(\pi^{m-1})$.
\end{ex}

We now state another formal corollary which utilizes the algebraic description of the quantity $\dTC_m(\pi)$ under some technical assumptions on $\pi^m$. 
%This is an analogue of a result in~\cite{KW}.

Let us denote $K(\pi,1)$ by $B\pi$. For the universal covering $q^{\pi}\colon E\pi \to B\pi$, the space $E\pi_{n+1}$ is homeomorphic to the space $E\pi\times_{\pi}\Delta^{\pi}_n$ obtained by the Borel construction, where $\Delta^{\pi}_n$ is the $n$-th skeleton of the simplex $\Delta^\pi$ with vertex set $\pi$. For any $m\ge 2$, it is easy to see that $E\pi^m/\pi\simeq E\pi^m\times_{\pi^m}(\pi^m/\pi)$, so that we have a fiber bundle $E\pi^m\times_{\pi^m} \Delta_n^{\pi^m/\pi}\to B\pi^m$ for each $n\ge 1$.

\begin{corollary}\label{cor: dTCD is cd}
For $X$ with $\pi_1(X)=\pi$, suppose that $\Wi X$ is $(k-1)$-connected and that $\cd(\pi) \leq k$. For $m\ge 2$, suppose $\cd(\pi^m)=n+1\ge 3$ with $H^{n+1}(B\pi^m;\A)\ne 0$ for an $\pi^m$-module $\A$. If the differential
\[
d_{n+1}\colon H^0\left(B\pi^m;H^{n}\left(\Delta_n^{\pi^m/\pi};\A\right)\right)\to H^{n+1}\left(B\pi^m;\A\right)
\]
is non-zero in the Serre spectral sequence for the fibration $E\pi^m\times_{\pi^m} \Delta_n^{\pi^m/\pi}\to B\pi^m$, then
\[
\dTC_m^{\D}(X) =  \TC_m^{\D}(\pi)=\cd(\pi^m).
\]
\end{corollary}

\begin{proof}
    By Theorem~\ref{thm:cdconstraintdTCD}, we have $\dTC^{\D}_m(X)=\dTC_m(\pi)$ due to the assumption on $\cd(\pi)$. Since the above differential $d_{n+1}$ is non-zero, we also have the equality $\dTC_m(\pi)=\cd(\pi^m)$ due to~\cite[Theorem 7.9]{KW}. Now, standard inequalities give
    \[
    \cd(\pi^m)=\dTC_m(\pi)=\dTC^{\D}_m(X)\le\TC_m^{\D}(X)\le\TC_m(\pi)\le\cd(\pi^m).
    \]
\end{proof}

The behavior of $\dTC^{\D}_m$ on spaces with finite fundamental group is rigid.
\begin{proposition}\label{prop:dTCD finite pi1}
If $X$ has $\pi_1(X)=\pi$ finite, then for each $m\ge 2$,
\[
\dTC_m^{\D}(X) \leq |\pi| - 1.
\]
\end{proposition}
\begin{proof}
    Proposition~\ref{prop: dtcD vs dtc} and~\eqref{four} imply
    \[
    \dTC^{\D}_m(X)\le\dTC^{\D}_m(\pi)=\dTC_m(\pi)\le |\pi|-1.
    \]
    Here, the last inequality is due to~\cite[Theorem 7.2]{KW}.
\end{proof}

Almost all sectional categorical invariants $\mathbbm{n}$, such as $\cat$, $\TC_m$, $\cat_1$, and $\TC_m^{\D}$, satisfy a product rule inequality $\mathbbm{n}(X\times Y)\le\mathbbm{n}(X)+\mathbbm{n}(Y)$ for ``nice" spaces $X$ and $Y$. Since the inception of distributional sectional categorical invariants $\text{d}\mathbbm{n}$, the following question has been around~(\cite[Question 3.4]{DJ}).

\begin{question}\label{ques: prod ineq}
    For nice spaces $X$ and $Y$, does $\text{d}\mathbbm{n}(X\times Y)\le\text{d}\mathbbm{n}(X)+\text{d}\mathbbm{n}(Y)$ hold? In particular, is $\text{d}\mathbbm{n}(X\times X)\le 2\hspace{0.5mm}\text{d}\mathbbm{n}(X)$ true?
\end{question}

The above product rule inequality is satisfied by some spaces and invariants $\text{d}\mathbbm{n}$, see, for example,~\cite{DJ,Ja1,JO}, and Corollaries~\ref{cor:cd3} and~\ref{cor:dTCD2}. However, quite surprisingly, the general answer to 
Question~\ref{ques: prod ineq} turns out to be \emph{NO} for 
$\text{d}\mathbbm{n}\in\{\dcat,\dTC_m,\dcat_1,\dTC^{\D}_m\}$, as our following example shows.

\begin{ex}\label{exam:pi1Z2}
Let $X$ be any space with $\pi_1(X)=\Z_2$. Then  $\dTC_m^{\D}(X) = 1$ by Proposition~\ref{prop:dTCD finite pi1} as $X$ is not simply connected. 

While $\dcat_1(X)=1$, we see that $\dcat_1(\Z_2\times \Z_2) = 3$ since $\Z_2 \times \Z_2$ is a $2$-group, see~\cite[Theorem 1.1(2)]{KW2}. In particular, we get
\[
\dcat(\Z_2\times\Z_2)=\dcat_1(\Z_2\times\Z_2)=3 > 2=2\dcat_1(\Z_2)=2\dcat(\Z_2), 
\]
thereby answering Question~\ref{ques: prod ineq} in the negative for $\dcat_1$ 
and $\dcat$.   Question~\ref{ques: prod ineq} is also answered in the 
negative for $\dTC_m^{\D}$ and $\dTC_m$ for each $m\ge 2$, because we get, using Proposition~\ref{prop: dtcD vs dcat1}, that
\[
\dTC_m^{\D}(\Z_2\times\Z_2) \ge \dcat_1(\Z_2^{2(m-1)}) =4m-5 > 2=2\dTC_m^{\D}(\Z_2).
\]
The same inequality as above holds for $\dTC_m$ in view of Proposition~\ref{prop: dtcD vs dtc}.
\end{ex}

\section{An upper bound for distributional category}\label{sec:an inequality}
In this section, we turn back to distributional category $(\dcat)$, with the primary aim of obtaining new examples of finite-dimensional CW complexes (in fact, closed manifolds) on which the distributional invariants \emph{do not} coincide with their classical counterparts. 

We know that $\dcat_1(X) \leq \dcat(X)$ for all $X$, with equality holding sometimes. We can also get an upper bound for $\dcat$ in terms of $\dcat_1$ and an important player in the very definition of $\dcat_1$, namely the universal cover. The best type of inequality giving a good approximation for $\dcat(X)$ from above would be an additive one. However, we know from Example~\ref{exam:pi1Z2} that a product inequality of the type $\dcat(X \times Y) \leq \dcat(X) + 
\dcat(Y)$ does not exist. So, the best we can do is obtain a multiplicative inequality that uses two quantities that are definitely less than (or equal to) $\dcat(X)$ itself.

\begin{theorem}\label{thm:general ineq}
    If $p\colon Y\to X$ is a covering map, then
    \[
    \dcat(X)\le (\dcat_1(X)+1)(\dcat(Y)+1)-1.
    \]
\end{theorem}
\begin{proof}
    Let $x_0\in X$ be the basepoint of $X$ and $y_0\in p^{-1}(x_0)$ be the basepoint of $Y$. Recall from Section~\ref{subsec: behavior covering} that $\Wi{X}=P_0(X)/[-]$ and $\Wi{Y}=P_0(Y)/\{-\}$, and $\theta\colon \Wi X\to\Wi Y$ is a homeomorphism which maps any $[\gamma]\in \Wi X$ to $\{\Wi{\gamma}\}\in \Wi Y$, where $\Wi{\gamma}\in P_0(Y)$ is the lift of $\gamma\in P_0(X)$ with respect to the point $\gamma(1)=x_0$. Recall also that the universal coverings $q^X\colon \Wi X\to X$ and $q^Y\colon \Wi Y\to Y$ satisfy $q^X[\alpha]=\alpha(0)$ and $q^Y\{\beta\}=\beta(0)$. Now, the homeomorphism $\theta$ 
satisfies $p\circ q^Y \circ \theta = q^X$, so we have a diagram
$$\xymatrix{
\Wi X \ar[rr]^-\theta \ar[dr]_-{q^X} & & \Wi Y \ar[dl]^-{p\circ q^Y} \\
& X &
} 
$$
that induces a map $\theta_{n+1}\colon \Wi X_{n+1} \to \Wi Y^p_{n+1}$, where $\Wi Y^p_{n+1}$
denotes the distributional construction on the covering $p \circ q^Y$. This means that an element
\[
\sum_{i=1}^{n+1} b_i \Wi{y}_i  \in \Wi Y^p_{n+1}
\]
has all $\Wi{y_i}$ in the same fiber of $p \circ q^Y$, but not necessarily in the same fiber
of $q^Y$. Essentially, we are taking points in the fiber of $p$ and then taking points in 
the fibers over these points. 
This explains the notation $\Wi Y^p_{n+1}$. Also, the map $q^Y$ induces a map $(q^Y_p)_{n+1} 
\colon \Wi Y^p_{n+1} \to Y_{n+1}$, where $Y_{n+1}$ is, as usual, the distributional construction applied to the covering $p\colon Y \to X$. 

Now, let $\dcat_1(X)=n$ and $\dcat(Y)=r$. We then have a section $\sigma\colon X\to \Wi X_{n+1}$ of $q^X_{n+1}$. For $x\in X$, suppose
    \[
    \sigma(x)=\sum_{i=1}^{n+1}a_i\Wi{x}_i \ \text{ with } \ q^X(\Wi{x}_i)=x.
    \]
    For each $1\le i \le n+1$, we can write $\Wi{x}_i=[\gamma_i]$ for some $\gamma_i\in P_0(X)$ 
such that $\gamma_i(0)=x$. Then, for $\theta_{n+1}\circ \sigma\colon X\to \Wi Y^p_{n+1}$, we have that
    \[
    \theta_{n+1}\circ \sigma(x) =\sum_{i=1}^{n+1}a_i\{\Wi{\gamma}_i\} \ \text{ with } \ p(\Wi{\gamma_i}(0))=x.
    \]
(Note that this element is not in $\Wi Y_{n+1}$ since the $\Wi{\gamma_i}$ are not
necessarily in the same fiber of $q^Y$.) 
For simplicity, let us write $\Wi{\gamma_i}(0)=y_i\in Y$. Then,
    \[
    (q^Y_p)_{n+1}\circ \theta_{n+1}\circ \sigma(x) =\sum_{i=1}^{n+1}a_iy_i \ \text{ with } \ p(y_i)=x.
    \]
    Since $\dcat(Y)=r$, the map $e^Y_{r+1}$ has a section $s\colon Y\to P_0(Y)_{r+1}$. 
For each $y_i$ as above, suppose
\[     
s(y_i)=\sum_{j=1}^{r+1}b_j\alpha_{ij} \ \text{ with } \ \alpha_{ij}(0)=y_i.       
\]
    Let $k=(n+1)(r+1)$. Suppose $\tau\colon P_0(Y)\to P_0(X)$ is the natural map defined 
as $\tau(\phi)=p\circ\phi$. This extends to a map $\tau_k\colon P_0(Y)_k \to P_0(X)_k$
since $p \circ e^Y = e^X \circ \tau$. Note that $\tau_k$ restricted to $P_0(Y)_{r+1}$
is $\tau_{r+1}$. Define a map $t\colon Y_{n+1}\to P_0(X)_k$ by
\[
t \left( \sum_{i=1}^{n+1}\lambda_iz_i \right) = \sum_{i=1}^{n+1}\lambda_i \tau_k\left(s(z_i)\right)
\]
for $z_i\in Y$. Then, feeding the measure $(q^Y_p)_{n+1}\circ \theta_{n+1}\circ \sigma(x) \in Y_{n+1}$ into $t$ yields
\[
t \left( \sum_{i=1}^{n+1}a_iy_i \right) = \sum_{i=1}^{n+1}a_i \tau_k\left(s(y_i)\right) = \sum_{i=1}^{n+1}\sum_{j=1}^{r+1}a_ib_j\left(p\circ\alpha_{ij}\right).
\]
Note that $e^Y(\alpha_{ij})=y_i$ for all $j$ and $e^X(p\circ\alpha_{ij})=x$ for all $i$ and $j$, So, the composition $t\circ (q^Y_p)_{n+1}\circ \theta_{n+1}\circ \sigma\colon X\to P_0(X)_k$ is a section of the fibration $e_k^X\colon P_0(X)_k\to X$. Therefore, $\dcat(X) \leq k-1 = (n+1)(r+1)-1$.
\end{proof}

The case of our particular interest is as follows.

\begin{corollary}\label{cor:ineq}
If $q\colon \widetilde X \to X$ denotes the universal covering map, 
then 
$$\dcat(X) \leq (\dcat_1(X) + 1)(\dcat(\widetilde X) + 1) - 1.$$
\end{corollary}

\begin{ex}\label{exam:Kpi1}
Suppose $X=K(\pi,1)$. Then $\widetilde X \simeq *$, so that $\dcat(\widetilde X)
= 0$. Hence, the inequality in Corollary~\ref{cor:ineq} reduces to $\dcat(X) \leq \dcat_1(X)$ and, 
together with the standard $\dcat_1(X) \leq \dcat(X)$, this recovers our 
equality $\dcat_1(X) = \dcat(X)$ for $X$ an Eilenberg--Mac~Lane space (\emph{cf.} Corollary~\ref{cor:classmap}). For instance, let $X$ be a closed surface of genus greater than $1$. Then 
$\dcat_1(X)=\cd(\pi_1(X))=2=\dcat(X)$ and the inequality is an equality. 
Hence, the inequality in Theorem~\ref{thm:general ineq} is sharp.
%in some cases.
\end{ex}

\begin{ex}\label{exam:1connected}
If $X$ is simply connected, then $\widetilde X = X$ and $\dcat_1(X)=0$ so
the inequality reduces to $\dcat(X) \leq \dcat(X)$, a tautology. 

Suppose $X = K \times Y$, where $K=K(\pi,1)$ and $Y$ is simply connected. 
Then $K$ is a retract of $X$, so $\dcat(K) = \dcat_1(K) \leq \dcat_1(X)$, while the projection $X \to K$ induces an isomorphism of fundamental
groups, so $\dcat_1(X) \leq \dcat(K)$. Hence, $\dcat_1(X) = \dcat(K)$. 
Also, since $Y$ is simply connected, we have $\widetilde X \simeq Y$. The 
inequality then becomes 
$$\dcat(K \times Y) \leq (\dcat(K) + 1)(\dcat(Y) + 1) - 1,$$
and this is the standard product inequality that is known at present.
\end{ex}

\subsection{Closed manifolds with $\dcat\ne \cat$}\label{subsec:dcat ne cat}

So far, the only known examples of finite-dimensional CW complexes $X$ satisfying the inequality $\dcat(X)<\cat(X)$ are the real projective spaces,~\cite{DJ,KW}. As we show in this subsection (and the next one), the crowning achievement of Corollary~\ref{cor:ineq} is that it gives us several new examples of well-studied closed manifolds whose distributional category is smaller than their LS-category.

\begin{ex}[Projective product spaces]\label{exam:projprods}
Let $S=\prod_{i=1}^r S^{n_i}$ be a product of $r$ spheres with $n_1 \leq \cdots \leq n_r$. Let $\Z_2$ act antipodally on each individual sphere $S^{n_i}$ and diagonally on the product $S$. Denote the quotient by this action 
by $P_{\bar n}$, where $\bar n = (n_1,\ldots,n_r)$. This quotient is 
called a \emph{projective product space} and it clearly generalizes the
ordinary real projective space. In~\cite{FV}, the LS-category of $P_{\bar n}$
was determined to be
$$\cat(P_{\bar n}) = n_1 + r - 1.$$
Of course, we have $\dcat(P_{\bar n}) \leq \cat(P_{\bar n})$, but in 
most cases, we can use the inequality of Corollary~\ref{cor:ineq} to 
get a better bound. First, note that since $\Z_2$ acts freely, $S$ is the
universal covering space for $P_{\bar n}$ (when $n_1 > 1$, which we assume). 
By the inequality~\cite[Theorem 3.7]{DJ} for covering spaces, we immediately obtain $r \leq 
\dcat(P_{\bar n})$ since $\dcat(S) = r$ (by using rational cup-length below and category above,~\cite[Section~6.1]{DJ}). But now we recognize that $\pi_1(P_{\bar n}) = \Z_2$, so we know by Corollary~\ref{cor:dcat1=1} that
$\dcat_1(P_{\bar n}) = 1$. Corollary~\ref{cor:ineq}
then gives
\[
\dcat(P_{\bar n}) \leq (\dcat_1(P_{\bar n}) + 1)(\dcat(S) + 1) - 1 =2r+1.
\]
Hence, we have $r \leq \dcat(P_{\bar n}) \leq 2r+1$. Therefore, anytime we have
$n_1 > r + 2$, the bound given by the inequality of Corollary~\ref{cor:ineq} is better and $\dcat(P_{\bar n})<\cat(P_{\bar n})$. For 
instance, if we consider $S = S^{10} \times S^{11} \times S^{12}$, then 
$$\cat(P_{\bar n}) = 12\ \text{ while }\ 3 \leq \dcat(P_{\bar n}) \leq 7.$$
Also, we note that the rational cup-length of $P_{\bar n}$ does not help here because
$$H^*(P_{\bar n};\Q) \cong H^*(S;\Q)^{\Z_2}.$$
But the ring $H^*(S;\Q)^{\Z_2}$ of fixed cohomology under the $\Z_2$-action is a subring of $H^*(S;\Q)$ and hence, its rational cup-length does not exceed $3$.
%since the antipodal map preserves orientation only for odd-dimensional spheres.
In any case, the projective product spaces $P_{\bar n}$ generalize real projective spaces by being closed manifolds with $\dcat(P_{\bar n}) < \cat(P_{\bar n})$. Of
course, by choosing appropriate integers $n_1$, this difference can be made arbitrarily large. 
\end{ex}

\begin{ex}[Lens spaces]\label{exam:lens spaces}
Let $p$ be a prime, $n\ge 2$ be an integer, and $q_i$ for $1\le i \le n$ be integers co-prime to $p$. Suppose $\Z_p$ acts on $S^{2n-1}\subset \C^n$ such that
\[
(z_1,\ldots,z_n)\mapsto \left(e^{2\pi iq_1/p}\hspace{0.3mm}z_1,\ldots,e^{2\pi iq_n/p}\hspace{0.3mm}z_n\right).
\]
Then the quotient of $S^{2n-1}$ by this free $\Z_p$ action is called a \emph{lens space}, denoted $L^{2n-1}(p;q_1,\ldots,q_n)$ --- we will denote it simply by $L^{2n-1}_p$. These spaces also generalize real projective spaces. The universal covering is $q\colon S^{2n-1}\to L^{2n-1}_p$ and $\pi_1(L^{2n-1}_p)=\Z_p$, irrespective of the choice of $q_i$ (hence our notation $L^{2n-1}_p$). It is known due to~\cite{Dr,KW2} that $\dcat(L^{\infty}_p)=\dcat(\Z_p)=p-1$. So, the classifying map $L^{2n-1}_p\to K(\Z_p,1)$ gives us $\dcat_1(L^{2n-1}_p)\le \dcat_1(\Z_p)=\dcat(\Z_p)=p-1$ by Proposition~\ref{prop:classmap} and Corollary~\ref{cor:classmap}. But now, Corollary~\ref{cor:ineq} implies
\[
\dcat(L^{2n-1}_p)\le(\dcat_1(L^{2n-1}_p)+1)(\dcat(S^{2n-1})+1)-1\le 2p-1.
\]
Since $\cat(L^{2n-1}_p)=2n-1$ (see, for instance,~\cite[Example 9.29]{CLOT}), whenever 
we have $n>p$, we obtain the strict inequality $\dcat(L^{2n-1}_p)<\cat(L^{2n-1}_p)$. Of course, this gap can be made arbitrarily large by taking large $n$. Hence, we have more examples of closed manifolds whose $\dcat$ and $\cat$ values differ. We also note that the rational cup-length of $L^{2n-1}_p$ is $1$ for each $n,p\ge 2$, so it is not helpful in estimating $\dcat(L^{2n-1}_p)$.
\end{ex}

\subsection{Closed manifolds with $\dTC_m\ne\TC_m$ and $\dTC^{\D}_m\ne \TC^{\D}_m$}\label{subsec: dTC ne TC} 
Using Examples~\ref{exam:projprods} and~\ref{exam:lens spaces}, we also get the first examples of finite-dimensional CW complexes $X$ (besides the real projective spaces) for which $\dTC_m(X)<\TC_m(X)$. 

\begin{ex}\label{exam: dtc and tc differ}
    For projective product spaces $P_{\bar n}$, where $\bar n=(n_1,\ldots,n_r)$ with $n_1\le\cdots\le n_r$, we have from Example~\ref{exam:projprods} an upper bound
    \[
    \dTC_m(P_{\bar n})\le\dcat(P_{\bar n}^m)\le (\dcat(P_{\bar n})+1)^m-1\le (2r+2)^m-1,
    \]
    see~\cite[Section 3]{Ja1} for the first two inequalities. Note that for any $m\ge 2$, we have $n_1+r-1=\cat(P_{\bar n})\le\TC_m(P_{\bar n})$. So, whenever $n_1>(2r+2)^m-r$, we see that $\dTC_m(P_{\bar n})<\TC_m(P_{\bar n})$. This gives us infinitely many new examples of closed manifolds having arbitrarily far apart $\dTC_m$ and $\TC_m$ values for each $m$.
\end{ex}

\begin{ex}\label{exam: dtc and tc differ-2}
    For lens spaces $L^{2n-1}_p=L^{2n-1}(p;q_1,\ldots,q_n)$, we have a universal upper bound 
    \[
    \dTC_m(L^{2n-1}_p)\le\dcat((L^{2n-1}_p)^m)\le (\dcat(L^{2n-1}_p)+1)^m-1\le (2p)^m-1
    \]
    due to Example~\ref{exam:lens spaces}. For any $m\ge 2$, we have $2n-1=\cat(L^{2n-1}_p)\le\TC_m(L^{2n-1}_p)$. So, whenever $n>2^{m-1}p^m$, we see that $\dTC_m(L^{2n-1}_p)<\TC_m(L^{2n-1}_p)$.
\end{ex}

At the cost of some arithmetic, Example~\ref{exam: dtc and tc differ-2} can be improved by appealing to the lower bounds of~\cite{Dau}.

Our arguments in Examples~\ref{exam:projprods} and~\ref{exam:lens spaces} also give the following new example that shows $\dcat_1\ne\cat_1$ and $\dTC^{\D}_m\ne\TC^{\D}_m$.

\begin{ex}\label{ex:dim3}
%Let us take $\dim(S)=3$ in the notations of Example~\ref{exam:projprods}, so that the only possibilities for $\bar n$ are $(1,1,1)$, $(1,2)$, and $3$. Since $\pi_1(P_{\bar n})=\Z_2$, we have from~\cite{GLGA} (see also Section~\ref{subsec:appl}) that $\cat_1(P_{\bar n})=3$. But then we get 
%\[ \dcat_1(P_{\bar n})=1<3=\cat_1(P_{\bar n}) \ \text{ and } \ \dTC_m^{\D}(P_{\bar n})=1<3\le\TC_m^{\D}(P_{\bar n}) \]
%for each $m\ge 2$ using Corollary~\ref{cor:dcat1=1}, Example~\ref{exam:pi1Z2}, and~\eqref{five}. Thus, $P_{(1,1,1)}$ and $P_{(1,2)}$ are closed $3$-manifolds different from $\R P^3$ whose classical and distributional one-categories are different, as well as whose $\dTC_m^{\D}$ and $\TC_m^{\D}$ values disagree.
%
Since the lens spaces $L^{2n-1}_p$ are essential manifolds, we have the equality $\cat_1(L^{2n-1}_p)=2n-1$ (\emph{cf.} Section~\ref{sec:essential}). Therefore, for any $m\ge 2$, when $n>2^{m-1}p^m$, Example~\ref{exam:lens spaces} and Proposition~\ref{prop: dtcD vs dcat1} give us
    \[
    \dcat_1(L^{2n-1}_p)\le 2p-1<2n-1=\cat_1(L^{2n-1}_p), \hspace{5mm} \text{and}
    \]
    \[
    \dTC^{\D}_m(L^{2n-1}_p)\le\dcat_1((L^{2n-1}_p)^m)\le (2p)^m-1<2n-1\le\TC^{\D}_m(L^{2n-1}_p),
    \]
    where the second inequality is obtained using $\dcat_1(X^m)\le (\dcat_1(X)+1)^m-1$, which is easy to show (along the lines of the proof of~\cite[Proposition~3.9]{Ja1}).
\end{ex}

\section*{Acknowledgment}
The authors thank the anonymous referee for various valuable suggestions, and, in particular, for pointing out a gap in the original proof of Theorem~\ref{th: after kw}.

\end{document}